\documentclass{amsart}
\usepackage{yhmath}  
\usepackage{amsmath,amsthm}
\usepackage{marvosym}
\usepackage{graphicx}
\usepackage{color}
\usepackage{epsfig}
\usepackage{psfrag}
\usepackage{enumerate}
\usepackage{amssymb}

\usepackage[all,cmtip]{xy}

\usepackage{fullpage}
\usepackage[active]{srcltx}

\theoremstyle{plain}
\newtheorem{theorem}{Theorem} 
\newtheorem{cor}[theorem]{Corollary}
\newtheorem{conj}[theorem]{Conjecture}

\newtheorem{lemma}[theorem]{Lemma}

\newtheorem*{question}{Question}
\newtheorem*{claim*}{Claim}

\theoremstyle{definition}

\newtheorem{remark}[theorem]{Remark}
\newtheorem{prob}[theorem]{Problem}

\newcommand{\Q}{\ensuremath{\mathbb{Q}}}

\newcommand{\Z}{\ensuremath{\mathbb{Z}}}

\newcommand{\RP}{\ensuremath{\mathbb{RP}}}

\newcommand{\hatQ}{\ensuremath{{\widehat{\Q}}}}

\newcommand{\A}{\ensuremath{{\mathcal A}}}
\newcommand{\B}{\ensuremath{{\mathcal B}}}
\newcommand{\comment}[1]{}

\newcommand{\bdry}{\ensuremath{\partial}}

\DeclareMathOperator{\sgn}{sgn}

\newcommand{\nbhd}{\ensuremath{\mathcal{N}}}

\newcommand{\tanglepq}[1]{
	\raisebox{-.5em}{
		\psfrag{A}[][]{\tiny{$#1$}}
		\includegraphics[height=1.5em]{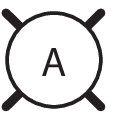}
		}
	}

\newcommand{\NE}{\mbox{\sc{ne}}}
\newcommand{\NW}{\mbox{\sc{nw}}}	

\newcommand{\SE}{\mbox{\sc{se}}}
\newcommand{\SW}{\mbox{\sc{sw}}}
\newcommand{\X}{\mbox{\sc{x}}}	

\newcommand{\NxNW}{\mbox{\sc{n}{\tiny x}\sc{nw}}}
\newcommand{\NxNE}{\mbox{\sc{n}{\tiny x}\sc{ne}}}
\newcommand{\ExNE}{\mbox{\sc{e}{\tiny x}\sc{ne}}}
\newcommand{\ExSE}{\mbox{\sc{e}{\tiny x}\sc{se}}}
\newcommand{\WxNW}{\mbox{\sc{w}{\tiny x}\sc{nw}}}
\newcommand{\WxSW}{\mbox{\sc{w}{\tiny x}\sc{sw}}}
\newcommand{\FxSW}{\mbox{\sc{f}{\tiny x}\sc{sw}}}
\newcommand{\FxNE}{\mbox{\sc{f}{\tiny x}\sc{ne}}}

\newcommand{\WSL}{\mbox{\em WSL}}
\newcommand{\BL}{\mbox{\em BL}}

\newcommand{\tanglesum}[2]{
	\raisebox{-.5em}{
		\psfrag{A}[Bc][Bc][.75]{\sc{#1}}\psfrag{B}[Bc][Bc][.75]{\sc{#2}}
		\includegraphics[height=1.5em]{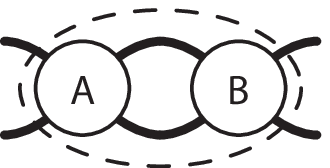}
		}
	}

\title{
On manifolds with multiple lens space filings}

\author{Kenneth L.\ Baker}
\address{
Department of Mathematics,
University of Miami,
PO Box 249085
Coral Gables, FL 33124-4250}
\email{k.baker@math.miami.edu}
\urladdr{http://math.miami.edu/\char126 kenken}

\author{Brandy Guntel Doleshal}
\address{Department of Mathematics and Statistics,
Sam Houston State University, 
Box 2206,
Huntsville, Texas 77341
}
\email{bdoleshal@shsu.edu}

\author{Neil Hoffman}
\address{Max Planck Institute for Mathematics,
Vivatsgasse 7,
53111 Bonn,
Germany}
\email{nhoffman@mpim-bonn.mpg.de}

\begin{document}

\begin{abstract}
An irreducible $3$--manifold with torus boundary either is a Seifert fibered space or admits at most three lens space fillings according to the Cyclic Surgery Theorem.  We examine the sharpness of this theorem by classifying the non-hyperbolic manifolds with more than one lens space filling, classifying the hyperbolic manifolds obtained by filling of the Minimally Twisted 5 Chain complement that have three lens space fillings, showing that the doubly primitive knots in $S^3$ and $S^1 \times S^2$ have no unexpected extra lens space surgery, and showing that the Figure Eight Knot Sister Manifold is the only non-Seifert fibered manifold with a properly embedded essential once-punctured torus and three lens space fillings.
\end{abstract}

\dedicatory{Dedicated to the 70th birthday of Professor Fico Gonz\'alez Acu\~na.}
\maketitle
\section{Introduction}
The Cyclic Surgery Theorem of Culler-Gordon-Luecke-Shalen \cite{CGLS} shows that an irreducible manifold whose boundary is a torus is either Seifert fibered or admits at most three lens space fillings.  In this article we examine the sharpness of this theorem with 
\begin{itemize}
\item Theorem~\ref{thm:nonhyperbolic} which determines the non-hyperbolic manifolds with at least two lens space fillings, 
\item Theorem~\ref{thm:MT5Ctriples} which classifies the hyperbolic manifolds that are the fillings of the exterior of the exterior of the Minimally Twisted 5 Chain link that have three lens space fillings, 
\item Theorem~\ref{thm:triplesofDPknots} which shows that the doubly primitive knots in $S^3$ and $S^1 \times S^2$ have no unexpected extra lens space surgery, and
\item Corollary~\ref{cor:figeightknottriple} which concludes that the Figure Eight Knot Sister manifold is the only non-Seifert fibered manifold with a properly embedded once-punctured torus and three lens space fillings.
\end{itemize}
Based on these, Conjecture~\ref{conj:threesurg} proposes a classification of hyperbolic manifolds with three lens space surgeries. Through these studies we have also observed a couple of behaviors that may be counter to one's expectations.
\begin{itemize}
\item Corollary~\ref{cor:notdp} shows that there are non-hyperbolic knots in lens spaces that have integral lens space surgeries which do not arise from Berge's doubly primitive construction.
\item  Corollary~\ref{cor:nonfibered} shows that there are many (non-primitive) non-fibered hyperbolic knots in lens spaces with non-trivial lens space surgeries.
\end{itemize}
Throughout we regard lens spaces as the closed, compact, connected, orientable $3$--manifolds that have a genus $1$ Heegaard splitting.  In particular, we include both $S^3$ and $S^1 \times S^2$ among the lens spaces.


\subsection{Basics of Dehn filling and Dehn surgery}
For more precise statements of our results, let us first recall some terminology about Dehn fillings and Dehn surgery.  A {\em slope} is an isotopy class of essential simple closed curves in a torus, and the {\em distance} $\Delta(r_1, r_2)$ between two slopes $r_1$ and $r_2$ is the minimum geometric intersection number among their representatives.  For a (compact, connected, oriented) $3$--manifold $M$ with torus boundary, {\em Dehn filling} $M$ along a slope $r$ in $\bdry M$ attaches a solid torus to $\bdry M$ so that $r$ bounds a meridional disk to produce both the manifold $M(r)$ and the knot $K \subset M(r)$ that is the core curve of the attached solid torus.  Given a knot $K$ in a $3$--manifold, removing a solid torus neighborhood $\nbhd(K)$ of $K$ and then Dehn filling along the resulting torus boundary is known as {\em Dehn surgery}, and the core curve of the attached solid torus is a knot called the {\em surgery dual} to $K$. 

A slope in $\bdry \nbhd(K)$ of a knot $K$ is a {\em framing} or a {\em longitude} if $K$ is isotopic to that slope in the solid torus, and a {\em framed knot} is the knot with a chosen framing.  A framed knot then has a natural {\em framed surgery} (Dehn surgery along the framing) and the surgery dual inherits a framing from the meridian of the original framed knot.  

A framing of a knot together with a meridian give a parametrization of the slopes in $\bdry \nbhd(K)$ as follows.  Choose an orientation of $K$, let $\mu$ be the meridian of $K$ oriented to link $K$ positively, and let $\lambda$ be the framing oriented to be parallel to $K$.  Then $([\mu], [\lambda])$ is a basis for $H_1(\bdry \nbhd(K);\Z)$ so that every oriented slope is represented by $p[\mu]+q[\lambda]$ for some pair of coprime integers $p,q$, and the unoriented slope is identified with $p/q \in \hatQ = \Q \cup \{1/0\}$.  For a null homologous knot, the framing by the Seifert surface is often used for this parametrization.

\subsection{Triples of Lens Space Fillings and Notable Knots, Links, and Manifolds}
\begin{figure}
\begin{tabular}{ccccccc}
\includegraphics[width=1.25in]{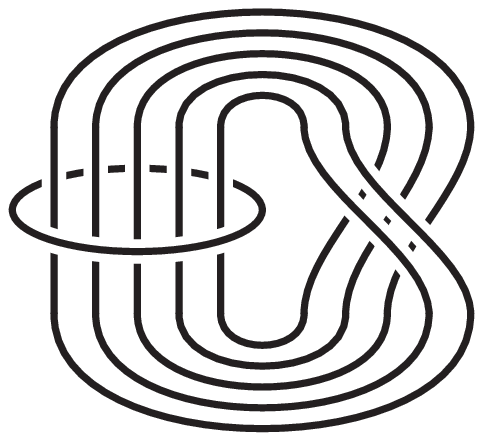} 
 \includegraphics[width=1.25in]{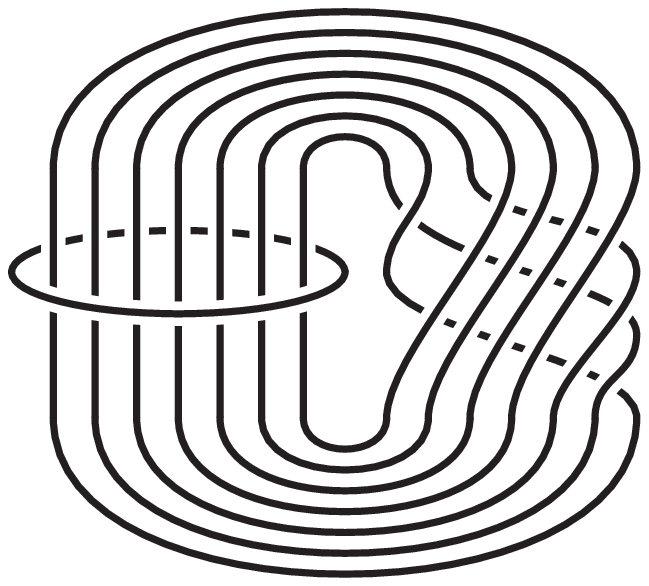}
\includegraphics[width=1.25in]{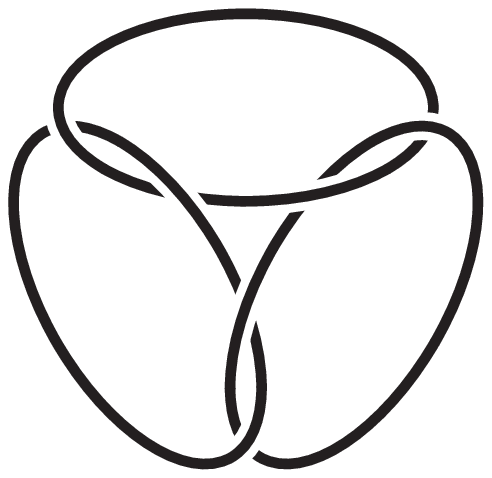}  
\includegraphics[width=1.25in]{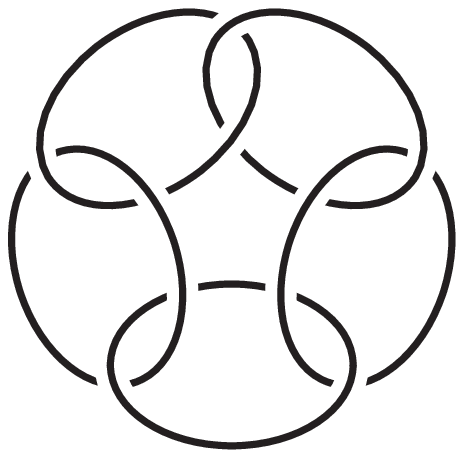} 
\end{tabular}
\caption{The Whitehead Sister Link, the Berge Link, the Magic Link, and the Minimally Twisted $5$ Chain.}
\label{fig:mainlinks}
\end{figure}
Figure~\ref{fig:mainlinks} shows four links. From left to right these are the Whitehead Sister Link which is the $(-2,3,8)$--pretzel link and denoted $\WSL$, the Berge Link denoted $\BL$, the Magic Link (which is the $(2,2,2)$--pretzel link), and the Minimally Twisted $5$ Chain which we abbreviate MT5C.  The exteriors of these links are the Whitehead Sister Manifold ($m125$ in the SnapPea Census \cite{snappy,callahanhildebrandweeks}), the Berge Manifold ($m202$), and the Magic Manifold ($s776$).  Following \cite{MPR}, we also let $M_3$ and $M_5$ denote exteriors of the last two.  (Note that in \cite{MPMagic2006} the mirror of $M_3$ is used and denoted $N$.)  For these links and their exteriors we use the parametrization of their slopes with the standard Seifert surface framing of each individual component to describe their Dehn surgeries and Dehn fillings. 

 Let $\WSL_{p/q}$ and $\BL_{p/q}$ denote the knot resulting from the Whitehead Sister Link or the Berge Link respectively by $p/q$ Dehn surgery on the unknotted component of those links.  It so happens that the knots $\WSL_{-1}$ and $\BL_{\infty}$  are both the $(-2,3,7)$--pretzel knot (with exterior $m016$) which is famously known to have two non-trivial lens space surgeries \cite{fintushelstern}.  Indeed $\WSL_{+1}$ is also a hyperbolic knot in $S^3$ with three lens space surgeries as are the knots $\BL_{1/n}$ for all integers $n$.  As one may observe from Tables A.6 and A.7 of \cite{MPMagic2006}, the knots $\WSL_{p}$ and $\BL_{p/q}$ are generically hyperbolic knots in lens spaces with two non-trivial lens space surgeries.   A more careful analysis of \cite{MPMagic2006} however reveals that the family of hyperbolic manifolds with three lens space fillings that come from the Whitehead Sister Manifold belong to a broader family.

\begin{figure}
\centering
\includegraphics[width=5in]{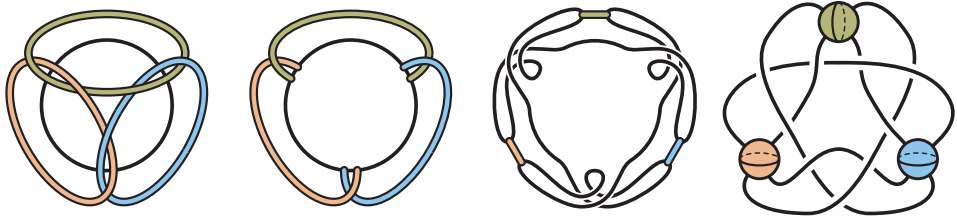}
\caption{The Magic Manifold is the double branched cover of the tangle $P_3$. }
\label{fig:magictangle}
\end{figure}

\begin{figure}
\centering
\psfrag{A}{\huge{$A'_{m,n}$}}
\psfrag{a}[c][c]{$P_3(3-\tfrac1m,2-\tfrac1n)$}
\psfrag{m}{\small{$m$}}
\psfrag{n}{\small{$n$}}
\psfrag{s}{$\simeq$}
\includegraphics[width=\textwidth]{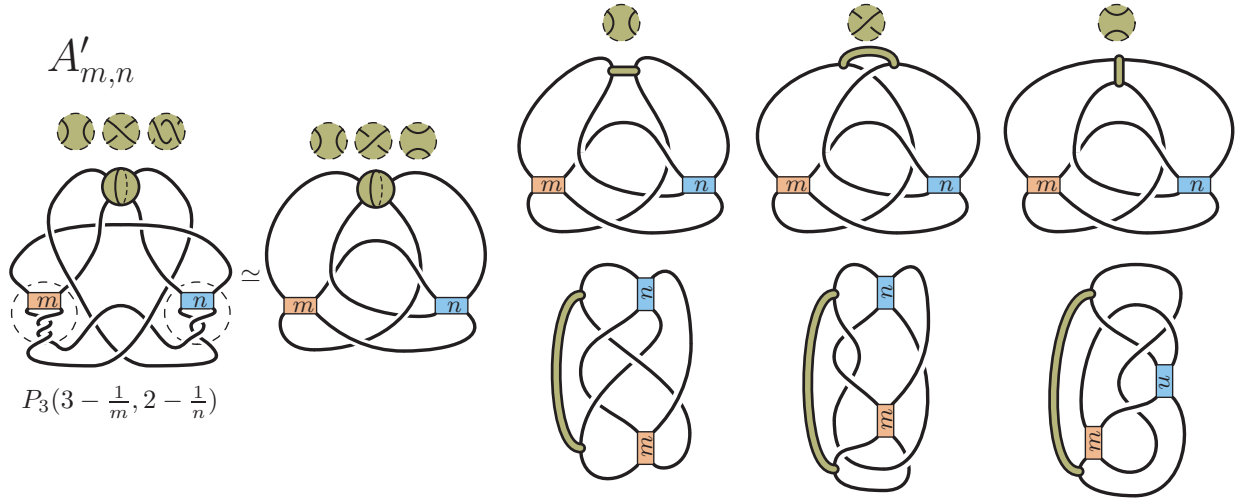}
\caption{The tangle $A'_{m,n}$ with its triple of two-bridge fillings.}
\label{fig:FamilyA}
\end{figure}
\begin{figure}
\centering
\psfrag{B}{\huge{$B'_{p/q}$}}
\psfrag{b}[c][c]{$P_3(5/2,p/q)$}
\psfrag{p}[c][c]{\small{$p/q$}}
\psfrag{q}[c][c]{\tiny{$\tfrac{p}{q}-2$}}
\psfrag{s}{$\simeq$}
\includegraphics[width=\textwidth]{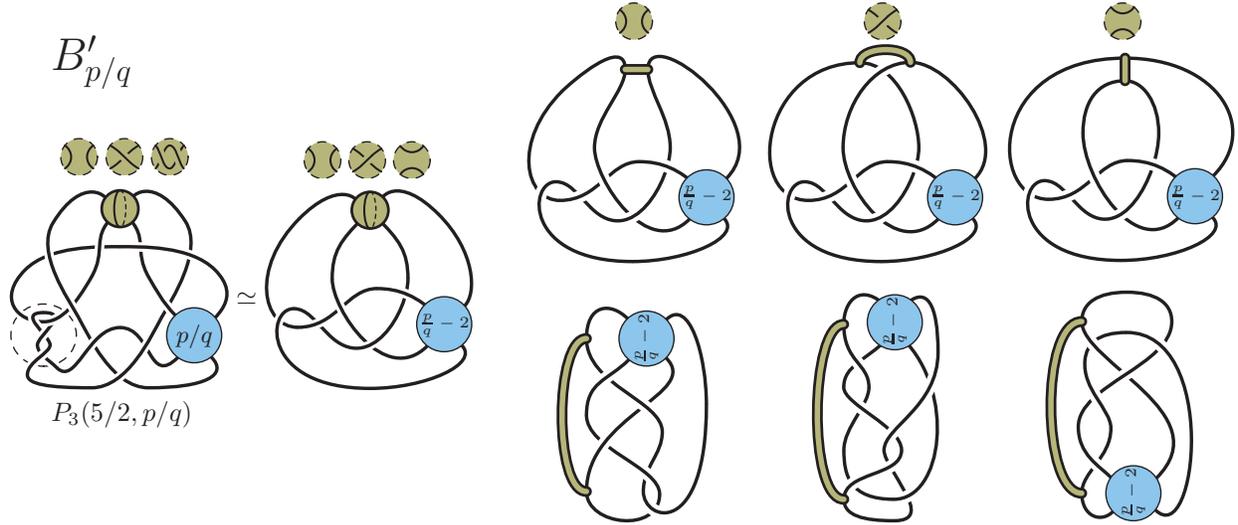}
\caption{The tangle $B'_{p/q}$ with its triple of two-bridge fillings.}
\label{fig:FamilyB}
\end{figure}
Figure~\ref{fig:magictangle} shows a tangle $P_3$ that is the quotient of $M_3$ by a strong involution.
Figures~\ref{fig:FamilyA} and \ref{fig:FamilyB} show the generic members of two families of tangles $\A' = \{A'_{m,n} \vert m,n\in\Z\}$ and  $\B' = \{B'_{p/q} \vert p/q \in \hatQ\}$ which are fillings of the tangle $P_3$.  Each admits three rational tangle fillings that produce two-bridge links as demonstrated in the last three columns of those figures.   Let $\A$ and $\B$ be the respective families of manifolds obtained as double branched covers of these tangles.   In terms of the Magic Manifold, these are explicitly the families of manifolds $\A = \{A_{m,n} = M_3(3-\tfrac1m,2-\tfrac1n) \vert m,n\in\Z\}$ with the three lens space fillings $\infty$, $1$, $2$ and $\B = \{B_{p/q} = M_3(5/2,p/q) \vert p/q \in \hatQ\}$ also with the three lens space fillings $\infty$, $1$, $2$.  
We note that the exterior of the knot $\WSL_p$ is homeomorphic to the manifold $A_{p+4,2}$ and the exterior of the knot $\BL_{r/s}$ is homeomorphic to the manifold $\B_{(4r+7s)/(r+2s)}$.



\begin{theorem}\label{thm:MT5Ctriples}
Let $M$ be a $3$--manifold obtained by filling four boundary components of $M_5$.  If $M$ admits three distinct lens space fillings, then either $M$ is the exterior of a trivial knot in a lens space, $M$ is the exterior of a torus knot in a lens space, or $M$ with its lens space fillings are homeomorphic to a member of family $\A$ or $\B$. 
\end{theorem}

We are informed that Audoux-Lecuona-Matignon-Paoluzzi-Roukema have obtained similar results in the case that one of the lens spaces is $S^3$ \cite{roukema-personalcommunication}.

\begin{proof}
If $M$ is non-hyperbolic, then it follows from Theorem~\ref{thm:nonhyperbolic} that $M$ is the exterior of either a trivial knot or a torus knot in a lens space.  If $M$ is hyperbolic then, using Martelli-Petronio-Roukema \cite{MPR} as a starting point, Theorem~\ref{thm:MT5CimpliesMM} shows that $M$ must be homeomorphic to a filling of $M_3$.   Theorem~\ref{thm:MMthreelensfillings} then uses the Martelli-Petroinio catalog of exceptional fillings of $M_3$ \cite{MPMagic2006} to show $M$ with its lens space fillings belongs to family $\A$ or $\B$.
\end{proof}

Due to Theorems~\ref{thm:MT5Ctriples} and Theorem~\ref{thm:triplesofDPknots}, the following conjecture seems reasonable.

\begin{conj}\label{conj:threesurg}
Let $M$ be a $3$--manifold whose boundary is a torus. 
If $M$ admits three lens space fillings along distinct slopes, then either $M$ is the exterior of a torus knot in a lens space or $M$ belongs to family $\A$ or $\B$.
\end{conj}

Here, a torus knot is any knot that embeds as an essential simple closed curve in a Heegaard torus of a lens space.  In a lens space, the trivial knot (i.e.\ the knot that bounds an embedded disk) is therefore included among the torus knots even though its exterior is not Seifert fibered in general.

\subsection{Doubly primitive knots and $(1,1)$--knots}

Following Berge \cite{TheBergeResult}, a knot $K$ in a $3$--manifold $M$ is {\em doubly primitive} if there is a genus $2$ Heegaard surface $\Sigma$ for $M$ containing $K$ such that to each side of $\Sigma$ there is a compressing disk whose boundary transversally intersects $K$ once.  We further say a framed knot $K$ is doubly primitive if there is such a Heegaard surface $\Sigma$ containing $K$ for which the induced surface slope is the framing.  Surgery on a doubly primitive framed knot along its doubly primitive framing produces a lens space.

Given a knot $K$ in a lens space $L(p,q)$ with a genus $1$ Heegaard splitting into two solid tori $V_\alpha$ and $V_\beta$, we say $K$ has a $(1,n)$--presentation for a positive integer $n$ if it may be expressed as the union of $n$ boundary-parallel arcs in each $V_\alpha$ and $V_\beta$.  We further say it is a $(1,n)$--knot, a {\em genus one $n$--bridge knot}, if it does not have a $(1,n-1)$--presentation; that is, the genus one bridge number of the knot is $n$. One may care to regard torus knots as $(1,0)$--knots, but it is convenient to include them among the $(1,1)$--knots.

The next lemma is a direct consequence of the definition of doubly primitive, perhaps implicit in Berge's work \cite{TheBergeResult}. Indeed the framed knot dual to any longitudinal surgery on a $(1,1)$--knot is a doubly primitive knot.  See also the Appendix of Saito \cite{saito2008dual}.

\begin{lemma}\label{lem:dp11}
A framed knot is doubly primitive if and only if its framed surgery dual is a $(1,1)$--knot. \qed
\end{lemma}


Since $(1,1)$--knots and doubly primitive knots have tunnel number one (i.e.\ their exteriors may be obtained by attaching a $2$--handle to a genus $2$ handlebody) they are strongly invertible.  Recall that strongly invertible knots may be viewed in the ``tangle setting'' as the lift of an embedded arc $a$ in the double branched cover of a link $L \subset S^3$ where $a \cap L = \bdry a$.  Indeed the exterior of the knot is the double branched cover of the tangle $(S^3 - \nbhd(a), L-\nbhd(a))$.  Because we largely work in this tangle setting, we find convenient to translate the notions of genus one $1$--bridge and doubly primitive.  This also makes Lemma~\ref{lem:dp11} rather transparent.

Say an embedded arc $a$ is {\em on} a link $L$ if $a \cap L = \bdry a$.
Then an arc $a$ on $L$ is {\em doubly primitive} if $L$ may be presented as a three-bridge link with bridge sphere $S$, $a \subset S$, and to each side of $S$ there is a bridge disk for $L$ that intersects $a$ in a single endpoint.
An arc $a$ on $L$ is {\em monotonic}  (or a $(1,1)$--arc) if $L$ may be presented as a two-bridge link with bridge sphere $S$, $|a \cap S|=1$ and to each side of $S$ there is a bridge disk for $L$ that contains a component of $a-S$.  This property is sometimes more easily verified by showing there is a height function $h \colon S^3 \to [-\infty, +\infty]$ such that $h|L$ has only two maxima and two minima and the extrema of $h|a$ occur at $\bdry a$.

\begin{theorem}\label{thm:ABdp11}
The knots that are the cores of the lens space fillings of the manifolds in families $\A$ and $\B$ are all $(1,1)$--knots and all doubly primitive knots with two doubly primitive framings.  
\end{theorem}

\begin{proof}
Figure~\ref{fig:FamilyA} shows that for each $m,n \in \Z$, the core arcs of the three two-bridge fillings of $A'_{m,n}$ are each monotonic.   Similarly, Figure~\ref{fig:FamilyB} shows that for each $p/q \in \hatQ$, the core arcs of the three two-bridge fillings of $B'_{p/q}$ are each monotonic.  Hence the corresponding knots in the double branched covers are all $(1,1)$--knots. 

Since each of these $(1,1)$--knots are dual to other $(1,1)$--knots by distance $1$ surgeries, Lemma~\ref{lem:dp11} applies to show that they are all doubly primitive as well with a doubly primitive framing for each of the other lens space surgeries.
\end{proof}

For the knots of family $\B$, we could have also appealed to Berge's proof that up to homeomorphism the exterior of the unknotted component of the Berge Link gives the unique hyperbolic knot in the solid torus with two non-trivial surgeries producing the solid torus  \cite{BergeSolidTori}.  Moreover, this knot is a $(1,1)$--knot in the solid torus.  The knots of family $\B$ may be obtained from the Berge Link by the various surgeries on the unknotted component.

In fact Gabai showed that any knot in the solid torus with a non-trivial solid torus surgery is a $(1,1)$--knot, and either the surgery is longitudinal or the knot is a torus knot \cite{GabaiSolidTori}.  Berge classified these knots \cite{BergeSolidTori}.  Collectively these knots are the {\em Berge-Gabai knots in a solid torus}.  Let us then say a knot in a lens space is a {\em Berge-Gabai knot} if some Heegaard torus for the lens space bounds a solid torus that contains the knot as a Berge-Gabai knot in a solid torus.  See also \cite{BBCW}.

\begin{theorem}\label{thm:bergegabaiknots}
The Berge-Gabai knots in lens spaces are all doubly primitive knots.  Every framed lens space surgery on a Berge-Gabai knots is given by a doubly primitive framing.
\end{theorem}
\begin{proof}
Every Berge-Gabai knot has at least one framed lens space surgery for which its surgery dual is also a Berge-Gabai knot.  This is due to its origin as a Berge-Gabai knot in a solid torus. Since the Berge-Gabai knots are all $(1,1)$--knots, Lemma~\ref{lem:dp11} implies that the Berge-Gabai knots are doubly primitive knots and this particular lens space surgery is done along a doubly primitive framing.  If a Berge-Gabai knot has another integral lens space surgery then either it is the trivial knot, its exterior is Seifert fibered and hence it is a torus knot, or it is hyperbolic by \cite{CGLS}.  (A priori, this other surgery need not manifest from a solid torus surgery on the corresponding Berge-Gabai knot in a solid torus.)   In the first two cases, these knots and their lens space surgery duals are all $(1,1)$--knots and so the theorem follows from Lemma~\ref{lem:dp11}.  For the last case, we note that it follows from the work of \cite{Baker2008MT5C} that the exteriors of the Berge-Gabai knots in solid tori may be obtained by three fillings of $M_5$.  Hence the Berge-Gabai knot exteriors are fillings of $M_5$.  Thus if a Berge-Gabai knot in a lens space is hyperbolic and has two integral lens surgeries then its exterior must be homeomorphic to a member of family $\A$ or $\B$ by Theorem~\ref{thm:MT5Ctriples}. Now apply Theorem~\ref{thm:ABdp11}.
\end{proof}

In light of Theorem~\ref{thm:bergegabaiknots}, one may suspect that a knot in a lens space with an integral lens space surgery must be a doubly primitive knot, but this is not so.  (Cf.\ Problem~1.9 \cite{BBCW}.)

\begin{theorem}\label{thm:largebridgenumber}
There are knots in lens spaces of arbitrarily large genus one bridge number that have integral lens space surgeries.
\end{theorem}
\begin{proof}
Let $U$ be an unknot in $L(n,1)$ with $|n|\geq2$ and let $T_{p,q}$ be the $(p,q)$--torus knot in $S^3$.  The genus $1$ bridge number of $U$ is $0$.  The genus $0$ bridge number of $T_{p,q}$ is $\min\{|p|,|q|\}$ by Schubert \cite{Schubert} (see also \cite{SchultensTorusKnot}).   By Doll, the genus one bridge number of the knot $K = U \# T_{p,q}$ in $L(n,1)$ is therefore at least $0 + \min\{|p|,|q|\}-2$.  By Theorem~\ref{thm:nonhyperbolic} below, this knot has an integral lens space surgery.  Clearly we may choose the torus knot $T_{p,q}$ so that the genus one bridge number of $K$ is as large as we wish.
\end{proof}

\begin{cor}\label{cor:notdp}
There are knots in lens spaces with integral lens space surgeries that are not doubly primitive.
\end{cor}
\begin{proof}
The surgery duals to the knots of Theorem~\ref{thm:largebridgenumber} have integral lens space surgeries but cannot be doubly primitive due to Lemma~\ref{lem:dp11}.
\end{proof}

\begin{remark}
The knots of Corollary~\ref{cor:notdp} are neither hyperbolic nor knots in $S^3$.
One may also obtain the result of Theorem~\ref{thm:largebridgenumber} for hyperbolic knots by using the work of \cite{BGL-integralsurgery}.  This applies to knots among certain families of Berge knots in $S^3$ and the famliy of cores of the $0$--filling in $M_3(0,n,4-n-1/m)$. Consequentially, their surgery duals give examples of knots with integral lens space surgeries but without doubly primitive presentations as in Corollary~\ref{cor:notdp}.  We have been informed that Bowman-Johnson are also able to show this for certain Berge knots in $S^3$ by other means \cite{Bowman-personalcommunication, Johnson-personalcommunication}.
\end{remark}

Nevertheless, all examples of knots in lens spaces with a non-trivial lens space surgery that we know are either doubly primitive or $(1,1)$.

\begin{conj}\label{conj:orders}
If knots $K_1 \subset L(p_1, q_1)$ and $K_2 \subset L(p_2,q_2)$ are longitudinal surgery duals, then up to reindexing $K_2$ is a $(1,1)$--knot and $p_2 \geq p_1$.
\end{conj}


\subsection{The doubly primitive knots in $S^3$ and $S^1\times S^2$.}
The doubly primitive knots in $S^3$ and in $S^1 \times S^2$ are known.  For $S^3$, Greene \cite{Greene2013Realization} showed Berge's list \cite{TheBergeResult} was complete.  For $S^1 \times S^2$, Cebanu \cite{Cebanu} has confirmed the list of Baker-Buck-Lecuona \cite{BBLSoneCrossStwo} is complete.  But note that these works do not explicitly address when two doubly primitive knots with different doubly primitive framings are actually isotopic as unframed knots.  Moreover, it is conceivable that a doubly primitive knot in some manifold may have a lens space surgery along an ``alternative'' framing that does not arise from a doubly primitive presentation.

\begin{theorem}\label{thm:triplesofDPknots}
Let $K$ be a doubly primitive knot with framing of slope $r$ in $S^3$ or $S^1 \times S^2$.  If $r\pm1$ surgery on $K$ is a lens space, then $K$ has a presentation as a doubly primitive knot with that framing.  Furthermore:

If $K$ is in $S^3$, then up to homeomorphism $K$ is $\WSL_{\pm1}$ or $\BL_{1/n}$ for $n\in \Z$.

If $K$ is in $S^1 \times S^2$, then up to homeomorphism $K$ is $\WSL_{0}$ or $\BL_{0}$.
\end{theorem}

\begin{proof}
The works  \cite{Baker2008MT5C, Baker2008GOFK} show that the doubly primitive knots in $S^3$ either (a) embed in the fiber of a genus one fibered knot (a trefoil or the figure eight knot) where the fiber gives the knot a doubly primitive framing or (b) admit a description as a filling of $M_5$.  The same holds for the doubly primitive knots in $S^1 \times S^2$: As described in \cite{BBLSoneCrossStwo}, they either embed in the fiber of a genus one fibered knot in $S^1 \times S^2$,  arise as a Berge-Gabai knot  (and hence admit a description as a filling of $M_5$), or are a ``sporadic'' knot.  The sporadic knots in $S^1 \times S^2$ are a slight modification of the sporadic knots in $S^3$, and it is readily apparent that this modification corresponds to an adjustment of the description as a filling of $M_5$ of the $S^3$--sporadic knots.
Therefore, by Theorem~\ref{thm:MT5Ctriples}, if a knot $K$ in $S^3$ or $S^1 \times S^2$ with a doubly primitive framing of slope $r$ admits a surgery description on the MT5C as well as a lens space surgery along the slope $r+1$ or $r-1$, then $K$ or its mirror belongs to family $\A$ or $\B$.

Let $K$ be a knot embedded in the fiber of a genus one fibered knot in either $S^3$ or $S^1 \times S^2$.     So now assume the fiber frames $K$ with slope $r$ and that $K$ admits another lens space surgery along the slope $r+1$ or $r-1$.  If $K \subset S^3$, then Theorem~\ref{thm:altGOFKsurg} shows that $K$ or its mirror is the knot obtained by either $+1$ or $-1$ surgery on the unknotted component of the Whitehead Sister Link.  If $K \subset S^1 \times S^2$ then Kadokami-Yamada show in Theorem~1.4 \cite{kadokamiyamada2012lens} that $K$ or its mirror is the knot obtained by $0$--surgery on the unknotted component of the Whitehead Sister Link.  Since these knots are surgery on the Whitehead Sister Link, they all belong to family $\A$.

By Theorem~\ref{thm:ABdp11}, the knots of families $\A$ and $\B$ are all $(1,1)$--knots, and thus any longitudinal lens space surgery slope is a doubly primitive framing.  Hence if knot $K$ in $S^3$ or $S^1\times S^2$ has two framings of distance $1$ that give surgeries to lens spaces,  then either both framings are doubly primitive framings for $K$ or neither are.

The remainder of the theorem now follows from the classification of knots in families $\A$ and $\B$.
\end{proof}

\begin{remark}
A key ingredient in the proof above is the determination of when a Dehn twist in the monodromy of a genus one fibered knot in $S^3$ or $S^1 \times S^2$ gives a genus one fibered knot in a lens space.  We do this for the genus one fibered knots in $S^3$ in Theorem~\ref{thm:altGOFKsurg} by determining which lens spaces that may be obtained by surgery on a knot in $S^3$ also contain a genus one fibered knot.  Kadokami-Yamada do this for the genus one fibered knots in $S^1 \times S^2$ by using Reidemeister torsion \cite{kadokamiyamada2012lens}. 
\end{remark}

\begin{question}
When do the monodromies of two genus one fibered knots in lens spaces differ by a single Dehn twist?  The curves along which such Dehn twists occur are knots with two non-trivial lens space surgeries.
\end{question}

\subsection{Non-hyperbolic knots in lens spaces with non-trivial lens space fillings.}

Let us say a knot is an {\em unknot} if its exterior is a solid torus and a {\em trivial knot} if is the boundary of a disk.  Only in $S^3$ are unknots and trivial knots equivalent.  Note that  an unknot in a lens space has a framed surgery that produces $S^3$ only if the lens space is homeomorphic to $L(n,1)$; moreover the framing is unique if $|n|\geq 2$.

For a torus knot, let $\lambda$ be the longitudinal framing given by the the Heegaard torus.  Recall that we include the trivial knots among the torus knots even though their exteriors are in general not Seifert fibered. Furthermore note that the framing $\lambda$ of a trivial knot as a torus knot agrees with the framing given by the disk it bounds.  Also recall that for a choice of orientation of the knot, we orient $\lambda$ as a parallel push-off and  let $\mu$ be a meridian oriented to link the knot positively.   

Lastly, let us note that the connected sum of two oriented framed knots produces a well-defined framed knot.  If one of the summands is reversible, then we may disregard the orientation condition.

\begin{theorem}\label{thm:nonhyperbolic}
Assume $K$ is a non-hyperbolic knot in a lens space with a non-trivial surgery yielding a lens space.  Then $K$ is either 
\begin{enumerate}
\item an unknot and surgery on $K$ along any slope produces a lens space,
\item a torus knot and surgery on $K$ along any slope distance $1$ from $\lambda$ produces a lens space, 
\item a $2\lambda\pm\mu$--cable of a torus knot and  $\pm1$ surgery on $K$ with respect to the framing by the cabling annulus produces a lens space,
\item a $p\lambda+q\mu$--cable of a torus knot on which $p/q$--surgery produces $S^3$ and  surgery on $K$ along the slope of the cabling annulus produces a lens space, or
\item the connected sum of a torus knot $T_{p,q}$ in $S^3$ and an unknot $U$ in $L(n,1)$ and surgery  of $K$ along the framing that arises in the connected sum from the framing $\lambda$ of $T_{p,q}$ and a framing on $U$ that gives $S^3$ by surgery produces a lens space.
\end{enumerate}
Moreover, in the last three cases, if the knot is not also itself a torus knot, then the stated surgery is the only non-trivial lens space surgery.  The last two cases are surgery dual.
\end{theorem}

\begin{remark}
We present this theorem because
we did not find any complete classification of non-hyperbolic manifolds with two lens space fillings in the literature.  In particular, we have not seen the last two cases of knots with their lens space surgeries mentioned before.  They are of particular interest due to Theorem~\ref{thm:largebridgenumber} and Corollary~\ref{cor:notdp}.
\end{remark}

\begin{remark}
We also note that in the last case of Theorem~\ref{thm:nonhyperbolic}, if $n=\pm1$ then the knot is simply a torus knot in $S^3$ with its integral lens space surgeries.   If $n=0$ then the knot is the unknot in $S^1\times S^2$ with its $S^3$ surgery.
\end{remark}

\begin{proof}
Let $K$ be a non-hyperbolic knot in a lens space $Y$.  Then its exterior $M=Y-\nbhd(K)$ is either reducible, Seifert fibered, or toroidal.

\smallskip
If $M$ is reducible, then  $K$ is contained in a ball in $Y$ and $Y \neq S^3$.  Since lens spaces are prime, if $K$ has a non-trivial surgery yielding a lens space, then $K$ must be the trivial knot by \cite{GordonLuecke1989}.  Every Dehn surgery on $K$ distance $1$ from the slope bounding a disk produces a homeomorphism of $Y$.

\smallskip
If $M$ is Seifert fibered, then either $K$ is a torus knot or a regular fiber of the Seifert fibration $M(-1;(n,1))$ of $Y=L(4n,2n-1)$  over the projective plane with one singular fiber.   If $K$ is the unknot, then every surgery on $K$ is a lens space.  If $K$ is a torus knot other than the unknot (or the trivial knot), then $K$ embeds in the Heegaard torus $T$ of $Y$ and $M\cap T=A$ is an essential annulus.  Then surgery on $K$ along any slopes distance $1$ from $\bdry A$  gives a lens space as apparent from the Seifert fibration of $M$ over the disk with two exceptional fibers, see e.g.\ \cite{DarcySumnersRationalTangleDistances}.  If $K$ is not a torus knot, then $K$ is the curve in the Klein bottle with annular complement in $\pm L(8,3)$ and $\pm1$ surgery on $K$ (with respect to the framing given by the Klein bottle) is $\mp L(8,3)$, \cite[Theorem 6.9]{BakerBuck-RationalSubtangleReplacement-AGT132013}.

\smallskip
If $M$ is toroidal, then there is an essential torus $T$ in $M$ which is either separating or, if $Y \cong S^1 \times S^2$, possibly non-separating. 
If $T$ is a non-separating essential torus in $M \subset Y= S^1 \times S^2$ then $K$ must be null-homologous.  If $K$ has a non-trivial lens space surgery, then since the non-separating torus will persist the resulting lens space must be $S^1 \times S^2$.  Gabai has shown that this implies $K$ must be a trivial knot \cite{gabai-Foliationsandthetopof3mfldsIII}.  Hence we may now assume $T$ separates.

 Let $M=N \cup_T M_K$ and $Y=N \cup_T Y_K$ where $\bdry N = T$ and $\bdry M_K = T \cup \bdry M$.   Furthermore, assume we have chosen $T$ so that $M_K$ is atoroidal.  Since $T$ must compress in $Y_K$,  $Y_K$ is homeomorphic to either $S^1 \times D^2$ or $S^1 \times D^2 \# Y$.   Assume non-trivial surgery on $K$ produces a lens space $Y'$ with dual knot $K'$. In particular, this surgery transforms $Y_K$ into the manifold $Y'_{K'}$ where $Y'=N \cup_T Y'_{K'}$ and $T$ compresses in $Y'_{K'}$.  As before $Y'_{K'}$ is homeomorphic to either $S^1 \times D^2$ or $S^1 \times D^2 \# Y'$.

If neither $Y_K$ nor $Y'_{K'}$ is a solid torus, then $N$ is the exterior of a non-trivial knot in $S^3$.   Since $T$ compresses in each $Y_K$ and $Y'_{K'}$, Scharlemann shows that either $Y_K$ is reducible or $K$ is isotopic into $T$ \cite{scharlemannProducingReducibleManifolds}.  In either case, $M$ is reducible and so $K$ is a trivial knot as shown above.  This contradicts that $M$ is toroidal.

Hence we may assume either $Y_K$ or $Y_K'$ is a solid torus; say $Y_K$ is a solid torus with core curve $C$.  Let $w$ denote the {\em winding number} of $K$ in $Y_K$, the minimum of the absolute values the algebraic intersection numbers of $K$ with the compressing disks of $Y_K$. Following \cite{BleilerLitherland1989} using \cite{GordonSatellite,GabaiSolidTori}, since $\bdry Y'_{K'}$ is compressible we may assume that $w \geq 2$.  Hence the slope of compression $Y'_{K'}$ on $T$ has distance greater than $1$ from the meridian of $Y_K$, and filling $N$ along this slope is either $S^3$ or $Y'$ according to whether $Y'_{K'}$ is reducible or a solid torus.  Then, by \cite{CGLS} and the Seifert fibered case above, it follows that $N$ is the exterior of a torus knot  in either $S^3$ or $Y'$ respectively. 

If $Y'_{K'} \cong S^1 \times D^2 \# Y'$ (with $Y' \not \cong S^3$) then $K$ is a cabled  knot and the surgery is along the slope of the cabling annulus by Corollary~4.4 of  \cite{scharlemannProducingReducibleManifolds}. Recall that $N$ is the exterior of a torus knot in $S^3$.  Hence $C$ is a torus knot in $Y$ with a surgery to $S^3$, $K=J(C)$ is a cable of $C$ along this surgery slope, and $Y'$ is obtained by surgery on $K$ along the cabling slope.  It follows that the surgery dual of this knot is the connected sum of an unknot in $L(n,1)$ and the torus knot in $S^3$ with exterior $N$.

If $Y'_{K'} \cong S^1 \times D^2$, then $K \subset Y_K$ is a Berge-Gabai knot in a solid torus and $K \subset Y$ is a satellite of a torus knot $C$ in $Y$.     Since $C$ is a torus knot and the lens space $Y' = Y'_{K'} \cup N$ may be viewed as surgery on $C$ , it must be that the meridian of $Y'_{K'}$ has distance $1$ from the framing $\lambda$ of $C$.  Since $Y_K$ is the closure of a regular neighborhood of $C$, we may frame it with $\lambda$.  
Observe that if $N'$ is the exterior of a torus knot in $S^3$ with framing $\lambda'$, then we may attach $Y_K$ so that $S^3=N' \cup Y_K$, the framings $\lambda$ and $\lambda'$ agree, and $N' \cup Y'_{K'}$ is a lens space.  In particular, we have used the  Berge-Gabai knot $K \subset Y_K$ to form a satellite of a torus knot in $S^3$ that has a lens space surgery.   Now we may appeal to \cite{BleilerLitherland1989} to conclude that $K$ is the $2\lambda \pm \mu$--cable of $C$ in the solid torus $Y_K$ and the surgery to $Y'_{K'}$ is along the slope of the cabling annulus.
\end{proof}

\subsection{Acknowlegedments}
The authors would like to thank Cameron Gordon,  Adam Lowrance,  Mario Eudave-Mu\~noz, and Fionntan Roukema for conversations about this work.

K.\ Baker thanks Tsuyoshi Kobayashi, Yo'av Rieck, and Nara Women's University for their hospitality during part of the writing of this article.  His work was partially supported by Simons Foundation grant \#209184 to Kenneth L.\ Baker.

N.\ Hoffman would like to thank the Max Planck Institute for Mathematics and Boston College for partially supporting this work.


\section{Notation and Conventions}

\subsection{Tangles}
We use the continued fraction expansion
\[p/q = [a_1, a_2, \dots, a_n] = a_1 -\cfrac{1}{a_2 - \cfrac{1}{a_2 - \cfrac{1}{ \ddots - \cfrac{1}{a_n}}}}\]
for an extended rational number $p/q \in \hat{\Q} = \Q \cup \{\infty\}$.  Typically we restrict to the case that the coefficients $a_i$ are integers, but it is  convenient to permit them to be rational numbers as well.

A {\em rational tangle} is a pair consisting of two arcs properly embedded in a ball such that the arcs are isotopic rel--$\bdry$ into the boundary of the ball.  Any two rational tangles are isotopic through rational tangles.  However, by choosing four fixed points around a circle on the boundary of the ball and restricting to rational tangles whose arcs have their endpoints at these four fixed points, the isotopy classes of these rational tangles are parametrized by $\hat{\Q}$.   Diagrammatically, represent a twist region between two strands of an embedded $1$--manifold by an oblong rectangle labeled with an integer $n$ if there are $|n|$ crossings and the twisting has handedness $\sgn(n)$.  Then we can build the rational tangle $p/q$, pictorially denoted $\tanglepq{p/q}$, from a continued fraction expansion $p/q = [a_1, a_2, \dots, a_n]$ with a sequence of twist regions assembled as depicted in Figure~\ref{fig:tanglelegend}. 
\begin{figure}
\centering
\psfrag{a}[][]{$a$}
\psfrag{b}{$b$}
\psfrag{c}[][]{$c$}
\psfrag{d}{$d$}
\psfrag{3}{$3$}
\psfrag{-13}[][]{$-\frac13$}
\psfrag{abc}[][]{$[a,b,c]$}
\psfrag{f}[][]{$[a,b,c,d]$}
\includegraphics[width=5in]{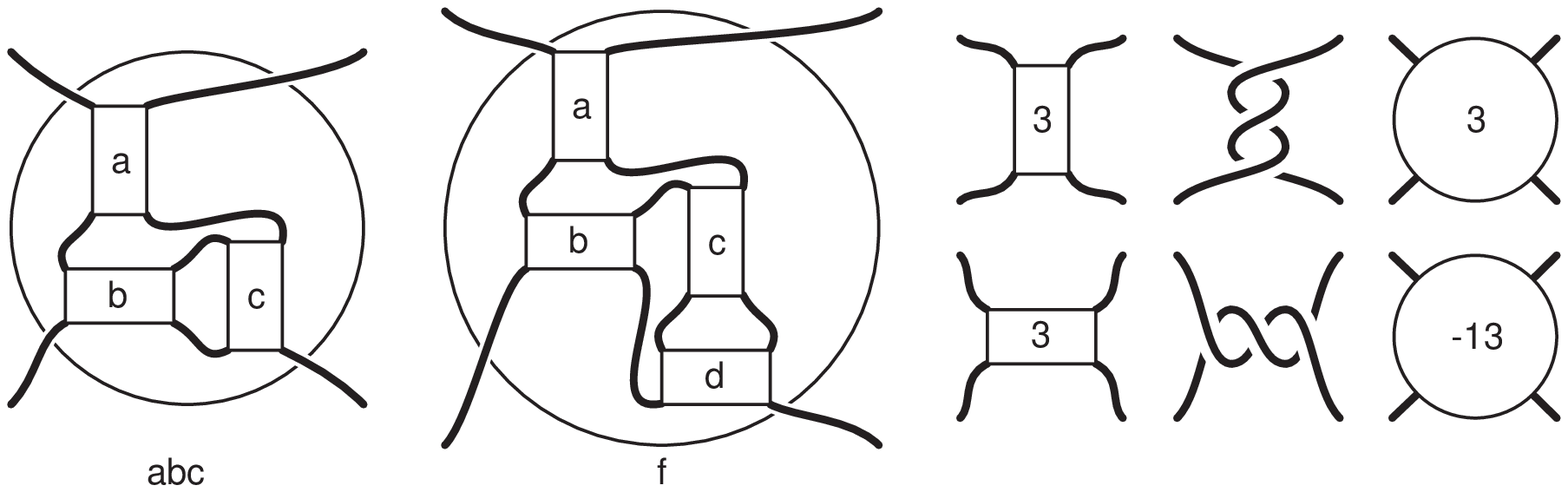}
\caption{\label{fig:tanglelegend}}
\end{figure}

In a sphere with four marked points, a slope is an isotopy class of embedded loops that separate the four points in pairs.  Note that for each rational tangle there is an isotopy class of properly embedded disks in the ball that separate the two arcs; the boundary of such a disk is a slope. Conversely, a slope determines a rational tangle.  A basis for these slopes is given by an ordered pair of two oriented slopes (up to overall reversal) with algebraic intersection number $+2$.   Figure~\ref{fig:tanglebasis} shows the standard choice of basis $\mu, \lambda$.  If the diagram of the tangle is in a different orientation than standard we will show at least the slope $\mu$.  Observe that the double cover of the sphere branched over the four points is a torus; the basis lifts to a basis for the torus, and slopes lift to slopes.  The double branched cover extends across any rational tangle filling the sphere defined by the slope to give a filling of the torus by a solid torus whose meridian is the lift of that slope.  This correspondence, and in particular the correspondence between replacing one rational tangle with another and Dehn surgery on knots, is often referred to as the Montesinos Trick \cite{montesinos}.

\begin{figure}
\centering
\psfrag{mer}[][]{$\mu$}
\psfrag{long}[][]{$\lambda$}
\includegraphics[height=1in]{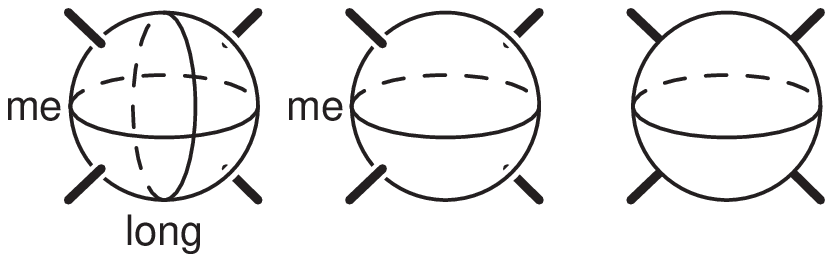}
\caption{}
\label{fig:tanglebasis}
\end{figure}

More generally, a tangle is the pair of $3$--manifold and a $1$--manifold properly embedded in it.  We will restrict attention to tangles in which the $3$--manifold is just $S^3$ minus a finite number of open balls and the $1$--manifold meets each boundary component in exactly four points.  For such a tangle $T$ let $\Sigma(T)$ denote its double branched cover.   We define a tangle $T$ to be {\em hyperbolic} if $\Sigma(T)$ is a hyperbolic manifold; otherwise $T$ is {\em non-hyperbolic}.  Here, a $3$--manifold is hyperbolic if its interior admits a complete hyperbolic metric with finite volume.

\subsection{Factoring fillings}
Given two tangles $A$ and $B$ such that $A \subset B$, we say a filling $\alpha$ of $A$ {\em factors through} $B$ if there exists a filling  $\beta$ of $B$ such that $A(\alpha)$ is orientation preserving homeomorphic to $B(\beta)$.  Similarly, given two $3$--manifolds $M$ and $N$ such that $M \subset N$, we say that a Dehn filling $\alpha$ of $M$ factors through $N$ if there exists a filling $\beta$ of $N$ such that $M(\alpha)$ is orientation preserving homeomorphic to $N(\beta)$.

\begin{figure}
\centering
\psfrag{a1}[B][r][.8]{$\alpha_1$}
\psfrag{a2}[B][r][.8]{$\alpha_2$}
\psfrag{a3}[B][r][.8]{$\alpha_3$}
\psfrag{a4}[B][r][.8]{$\alpha_4$}
\psfrag{a5}[B][r][.8]{$\alpha_5$}
\includegraphics[width=.5\textwidth]{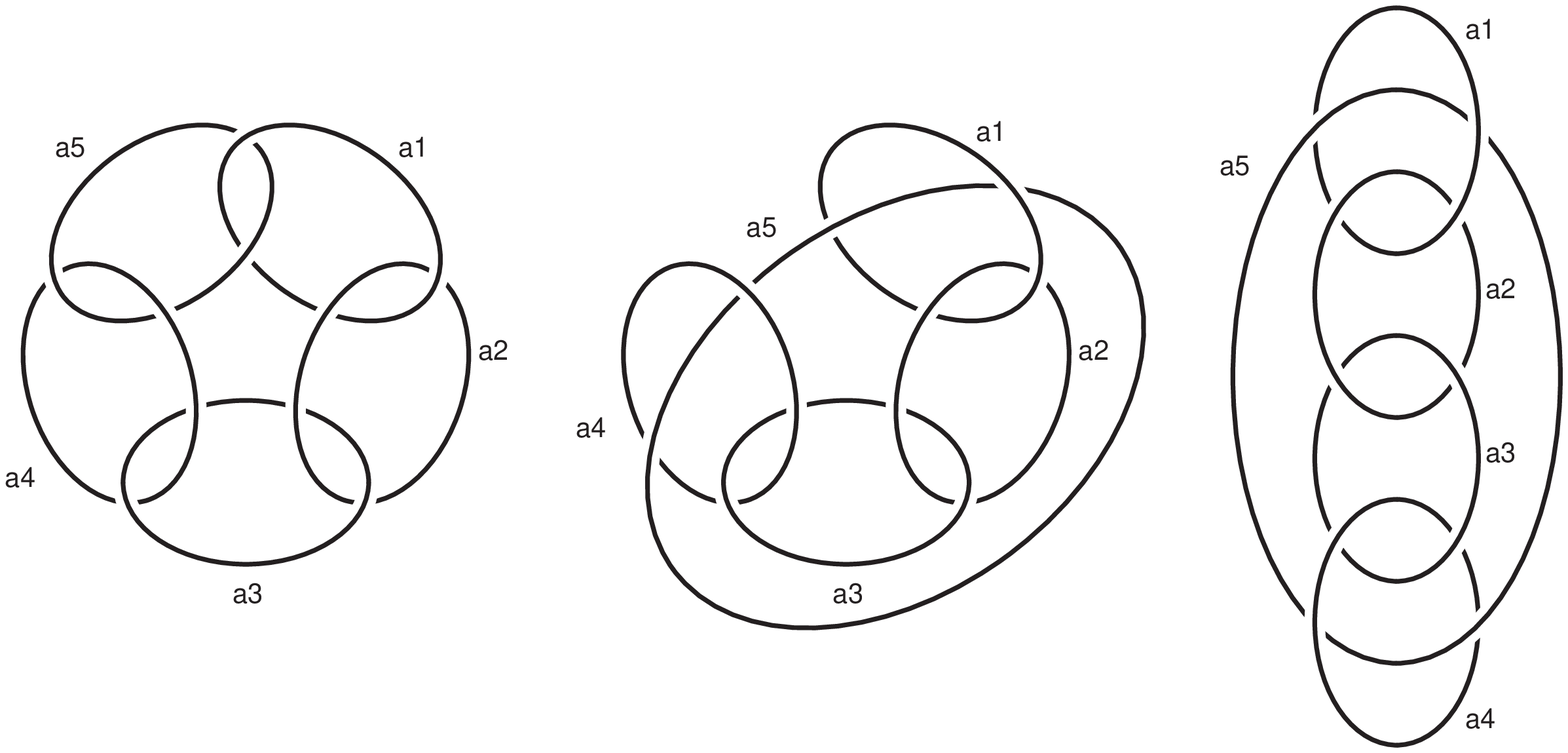}
\caption{}
\label{fig:mt5cisotopy}
\end{figure}

\subsection{MT5C and the Pentangle}
Figure~\ref{fig:mt5cisotopy} shows an isotopy of the Minimally Twisted 5 Chain (MT5C) from its common form to a form admitting a strong involution along a vertical axis.  Figure~\ref{fig:chaintopentangle} shows the quotient of its exterior $M_5$ under this involution and the subsequent isotopy into the form of the Pentangle $P_5$.  (In the first three frames, the top and the bottom of the axis join at the point at $\infty$.)  This quotient and isotopy follows the standard meridian-longitude basis for each component of $\bdry M_5$ coming from its identification as the exterior of the MT5C of Figure~\ref{fig:mt5cisotopy} to the induced basis for the quotient tangle and compares it with the standard basis given by this diagram of $P_5$ in the last frame.  
We will always take $P_5$ with this basis.   Figure~\ref{fig:chainpent} facilitates the translation between the bases of the boundary components of $M_5$ and $P_5$.  In particular, a sequence of Dehn fillings $(\alpha_1, \alpha_2, \alpha_3, \alpha_4, \alpha_5)$ of $M_5$ corresponds to a sequence of rational tangle fillings $(\NW, \NE, \SW, \SE, \X)$ of $P_5$ as follows:
\begin{align}
M_5(\alpha_1, \alpha_2, \alpha_3, \alpha_4, \alpha_5) &= \Sigma (P(\alpha_2,1-\tfrac{1}{\alpha_1},1-\tfrac{1}{\alpha_4},\alpha_3,\alpha_5-1)) \label{eqn:m5top5}\\
M_5(\tfrac{1}{1-\NE},\NW, \SE, \tfrac{1}{1-\SW}, \X+1) &= \Sigma(P(\NW, \NE, \SW, \SE, \X))
\end{align}

\begin{figure}
\centering
\psfrag{m1}[l][l][.8]{$\mu_1$}
\psfrag{m2}[l][l][.8]{$\mu_2$}
\psfrag{m3}[l][l][.8]{$\mu_3$}
\psfrag{m4}[l][l][.8]{$\mu_4$}
\psfrag{m5}[l][l][.8]{$\mu_5$}
\psfrag{l1}[B][][.8]{$\lambda_1$}
\psfrag{l2}[B][][.8]{$\lambda_2$}
\psfrag{l3}[B][][.8]{$\lambda_3$}
\psfrag{l4}[B][][.8]{$\lambda_4$}
\psfrag{l5}[B][][.8]{$\lambda_5$}
\psfrag{pm1}[][r][.8]{$\mu_{\mbox{\tiny \sc ne}}$}
\psfrag{pm2}[][r][.8]{$\mu_{\mbox{\tiny \sc nw}}$}
\psfrag{pm3}[][r][.8]{$\mu_{\mbox{\tiny \sc se}}$}
\psfrag{pm4}[][r][.8]{$\mu_{\mbox{\tiny \sc sw}}$}
\psfrag{pm5}[][r][.8]{$\mu_{\mbox{\tiny \sc x}}$}
\psfrag{pl1}[B][][.8]{$\lambda_{\mbox{\tiny \sc ne}}$}
\psfrag{pl2}[B][][.8]{$\lambda_{\mbox{\tiny \sc nw}}$}
\psfrag{pl3}[B][][.8]{$\lambda_{\mbox{\tiny \sc se}}$}
\psfrag{pl4}[B][][.8]{$\lambda_{\mbox{\tiny \sc sw}}$}
\psfrag{pl5}[B][][.8]{$\lambda_{\mbox{\tiny \sc x}}$}
\includegraphics[width=.9\textwidth]{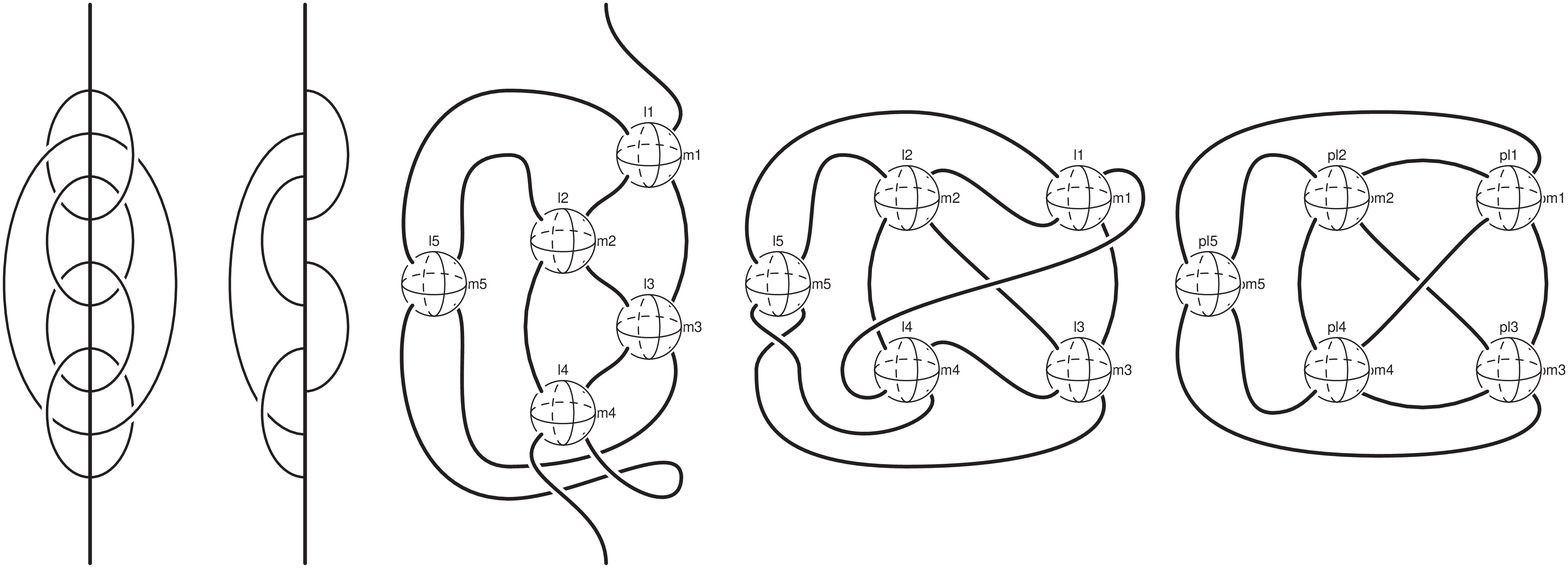}
\caption{}
\label{fig:chaintopentangle}
\end{figure}

\begin{figure}
\centering
\psfrag{a1}[B][r][.8]{$\alpha_1$}
\psfrag{a2}[B][r][.8]{$\alpha_2$}
\psfrag{a3}[B][r][.8]{$\alpha_3$}
\psfrag{a4}[B][r][.8]{$\alpha_4$}
\psfrag{a5}[B][r][.8]{$\alpha_5$}
\psfrag{b1}[][][.7]{$1\!\!-\!\!\tfrac{1}{\alpha_1}$}
\psfrag{b2}[][][.8]{$\alpha_2$}
\psfrag{b3}[][][.7]{$1\!\!-\!\!\tfrac{1}{\alpha_4}$}
\psfrag{b4}[][][.8]{$\alpha_3$}
\psfrag{b5}[][][.7]{$\alpha_5\!\!-\!\!1$}
\includegraphics[width=.5\textwidth]{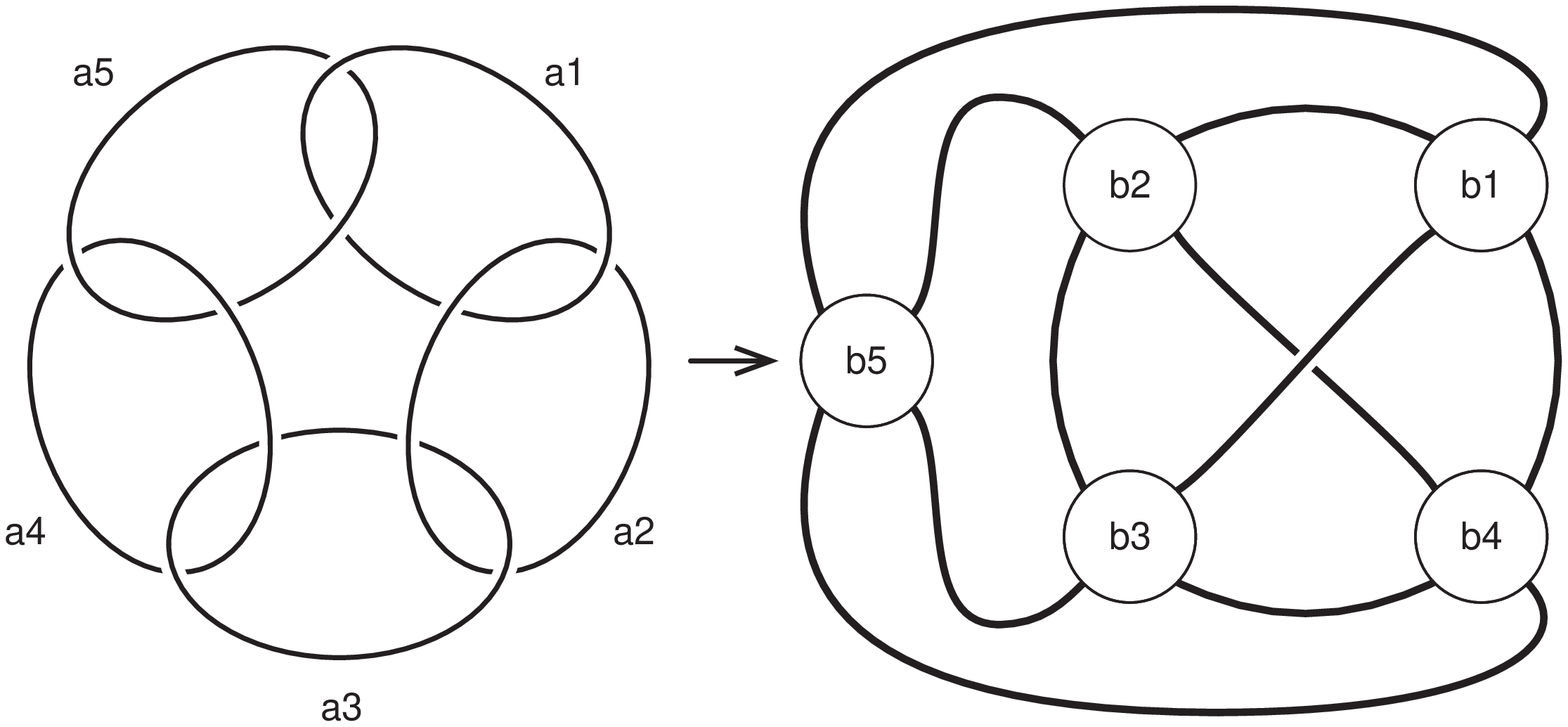}
\caption{}
\label{fig:chainpent}
\end{figure}

The pentangle admits three obvious involutions that preserve the fifth boundary component and its slopes while pairs of the other four:
\begin{align}
(\NW, \NE, \SW, \SE, \X) &\mapsto (\NE, \NW, \SE, \SW, \X)  \label{eqn:rightleft}\\
(\NW, \NE, \SW, \SE, \X) &\mapsto (\SW, \SE, \NW, \NE, \X) \label{eqn:topbot}\\
(\NW, \NE, \SW, \SE, \X) &\mapsto (\SE, \SW, \NE, \NW, \X) \label{eqn:frontback}
\end{align}
There is also an order $3$ symmetry illustrated in Figure~\ref{fig:order3rot} that fixes two boundary components (though altering their slopes) while permuting the other three.  The shaded triangle helps suggest the action of this symmetry, the slopes $(\NW=1, \SW=\infty, \NE=0)$ are cyclically permuted.  This rotation gives the following map:
\begin{align}
(\NW, \NE, \SW, \SE, \X) &\mapsto (\tfrac{1}{1-\NE}  , \tfrac{1}{1-\SW}  ,  \tfrac{1}{1-\NW}, \tfrac{1}{1-\SE},  \tfrac{-1}{1+\X}) \label{eqn:p5order3}
\end{align}

\begin{figure}
\centering
\psfrag{m1}[l][l][.8]{$\mu_1$}
\psfrag{m2}[l][l][.8]{$\mu_2$}
\psfrag{m3}[l][l][.8]{$\mu_3$}
\psfrag{m4}[l][l][.8]{$\mu_4$}
\psfrag{m5}[l][l][.8]{$\mu_5$}
\psfrag{l1}[B][][.8]{$\lambda_1$}
\psfrag{l2}[B][][.8]{$\lambda_2$}
\psfrag{l3}[B][][.8]{$\lambda_3$}
\psfrag{l4}[B][][.8]{$\lambda_4$}
\psfrag{l5}[B][][.8]{$\lambda_5$}
\psfrag{pm1}[][r][.8]{$\mu_{\mbox{\tiny \sc ne}}$}
\psfrag{pm2}[][r][.8]{$\mu_{\mbox{\tiny \sc nw}}$}
\psfrag{pm3}[][r][.8]{$\mu_{\mbox{\tiny \sc se}}$}
\psfrag{pm4}[][r][.8]{$\mu_{\mbox{\tiny \sc sw}}$}
\psfrag{pm5}[][r][.8]{$\mu_{\mbox{\tiny \sc x}}$}
\psfrag{pl1}[B][][.8]{$\lambda_{\mbox{\tiny \sc ne}}$}
\psfrag{pl2}[B][][.8]{$\lambda_{\mbox{\tiny \sc nw}}$}
\psfrag{pl3}[B][][.8]{$\lambda_{\mbox{\tiny \sc se}}$}
\psfrag{pl4}[B][][.8]{$\lambda_{\mbox{\tiny \sc sw}}$}
\psfrag{pl5}[B][][.8]{$\lambda_{\mbox{\tiny \sc x}}$}
\includegraphics[height=1.5in]{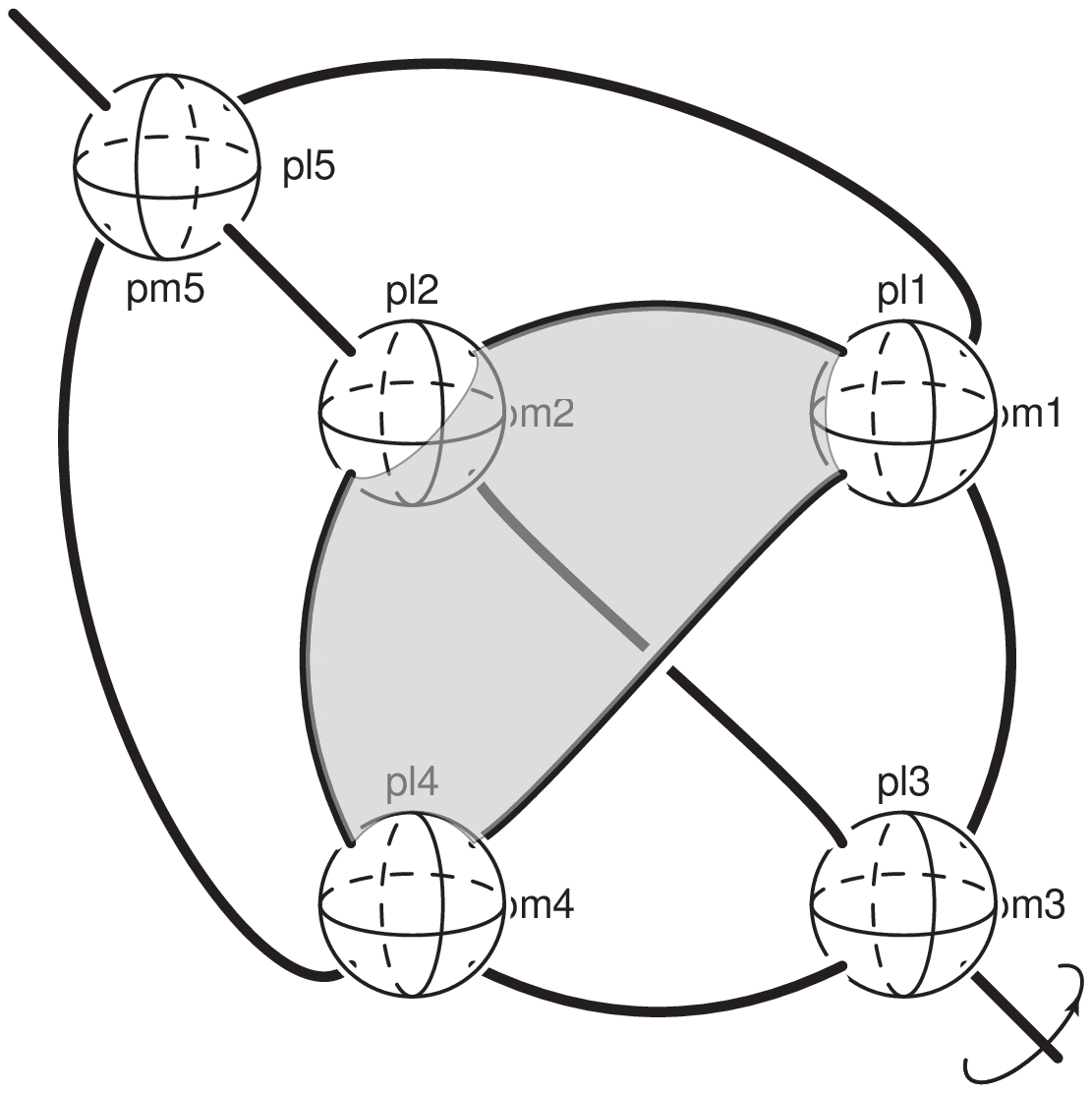}
\label{fig:order3rot}
\end{figure}

Rotating the figure a quarter turn and then taking the mirror provides another symmetry:
\begin{align}
(\NW,\NE,\SW,\SE,\X) & \mapsto (\tfrac{1}{\NE}, \tfrac{1}{\NE}, \tfrac{1}{\SE}, \tfrac{1}{\NW},\tfrac{1}{\SW}, \tfrac{1}{\X}) \label{eqn:mirrorsym}
\end{align}

\section{On lens space fillings of the MT5C}

Let $M_5$ denote the exterior of the MT5C. Let $M_3$ denote the magic manifold, i.e.\ the exterior of the $(2,2,2)$--pretzel link.
These have the quotient tangles $P_5$ and $P_3$.

\begin{lemma}\label{lem:nonhypfilling}
A filling  $\alpha = (\NW,\NE,\SW,\SE)$ of $P_5$ is non-hyperbolic if one of the following holds:
\begin{itemize}
\item $\NW, \NE, \SW,$ or $\SE \in \{0,1,\infty\}$
\item $\{\NW, \NE\}$ or $\{\SW, \SE\}  \in \{ \{-1,2\},\{\tfrac12,\tfrac12\} \}$
\item $\{\NW, \SW\}$ or $\{\NE, \SE\} \in \{ \{-1, \tfrac12\}, \{2,2\} \}$
\item $\{\NW, \SE\}$ or $\{\SW, \NE\} \in \{ \{\tfrac12, 2\}, \{-1,-1\} \}$
\end{itemize}
\end{lemma}

\begin{lemma}\label{lem:magictanglefilling}
A filling  $\alpha = (\NW,\NE,\SW,\SE)$ of $P_5$ factors through $P_3$ if one of the following holds:

\begin{itemize}
\item $\{\NW, \NE\}$ or $\{\SW, \SE\} \in \{ \{2,-2\}, \{-1,\tfrac32\}, \{\tfrac12,\tfrac13\}, 
 \{2,\tfrac12\}, \{-1,-1\} \}$

\item $\{\NW,\SW\}$ or $\{\NE, \SE\} \in \{ \{-1,\tfrac13\}, \{\tfrac12,-2\},\{2,\tfrac32\},
 \{-1, 2\}, \{\tfrac12, \tfrac12\} \}$

\item $\{\NW, \SE\}$ or $\{\NE, \SW\} \in \{ \{\tfrac12, \tfrac32\}, \{2,\tfrac13\}, \{-1,-2\},
 \{\tfrac12,-1\}, \{-2,-2\} \}$
\end{itemize}

It factors through the mirror of $P_3$ if one of the following holds:
\begin{itemize}
\item $\{\NW, \NE\}$ or $\{\SW, \SE\} \in \{\{-1,3\}, \{2,-\tfrac12\}, \{\tfrac12,\tfrac23\}, \{-1,\tfrac12\}, \{2,2\} \}$

\item $\{\NW,\SW\}$ or $\{\NE, \SE\} \in \{ \{\tfrac12, -\tfrac12\}, \{-1, \tfrac23\}, \{2,3\}, \{\tfrac12, 2\}, \{-1,-1\}  \}$

\item $\{\NW, \SE\}$ or $\{\NE, \SW\} \in \{ \{2,\tfrac23\}, \{\tfrac12,3\}, \{-1,-\tfrac12\}, \{2,-1\}, \{\tfrac12,\tfrac12\} \}$ \end{itemize}
\end{lemma}

\begin{proof}
Check that the five fillings $(2,2)$, $(-1,\tfrac32)$, $(\tfrac12,\tfrac13)$, $(2,\tfrac12)$, and $(-1,-1)$ for $(\NW, \NE)$ in $P_5$ do indeed produce $P_3$.  Then apply the symmetries of Equations (\ref{eqn:rightleft}),(\ref{eqn:topbot}), and (\ref{eqn:frontback}) to obtain all the fillings of the first line. The next two lines of fillings that yield $P_3$ are produced by applying Equation~(\ref{eqn:p5order3}) to the first fillings of the first line.
Thereafter Equation~(\ref{eqn:mirrorsym}) gives the relation between those fillings that factor through $P_3$ and those that factor through the mirror of $P_3$.
\end{proof}

\subsection{One-cusped hyperbolic fillings of $M_5$ with three lens space fillings}

The objective of this section is the following theorem which has the subsequent theorem as an immediate corollary.
\begin{theorem}\label{thm:fillingsimplifies}
If a filling $\alpha = (\NW, \NE, \SW, \SE)$ of $P_5$ admits three two-bridge fillings then $P_5(\alpha)$  either is non-hyperbolic or factors through $P_3$ or its mirror.
\end{theorem}

\begin{theorem}\label{thm:MT5CimpliesMM}
If a one-cusped hyperbolic manifold $M$ is obtained from a filling of $M_5$ and admits three lens space fillings, then $M$ may be obtained from a filling of $M_3$.
\end{theorem}

\begin{proof}
Observe that $M_5 =\Sigma(P_5)$ and lens spaces are double branched covers of two-bridge links.  Then apply Theorem~\ref{thm:fillingsimplifies}.
\end{proof}

Say a filling $\alpha = (\NW, \NE, \SW, \SE)$ of $P_5$ {\em simplifies} if either it is non-hyperbolic or it factors through $P_3$ or its mirror.  Thus the theorem asserts that any filling $\alpha$ of $P_5$ with three two-bridge fillings must simplify.  To prove this theorem we rely heavily on Lemmas~\ref{lem:nonhypfilling} and \ref{lem:magictanglefilling} which tell us when a filling simplifies. Note that the lemmas use unordered pairs, so $\{\NW, \NW\} = \{-1,2\}$ means $(\NW, \NE) =(-1,2)$ or $(2,-1)$ for example.

\begin{proof}[Proof of Theorem~\ref{thm:fillingsimplifies}]
Assume the filling $\alpha=(\NW, \NE, \SW,\SE)$ of $P_5$ does not simplify.  In particular, $\Sigma(P_5(\alpha)) = M_5(\tilde{\alpha})$ is hyperbolic and not homeomorphic to $M_3$.   If there are three choices for $\X$ such that $P(\alpha, \X)$ is two-bridge for each, then these fillings lift to three fillings of $\Sigma(P_5(\alpha)) = M_5(\tilde{\alpha})$ that produce lens spaces.  By \cite{CGLS}, these three fillings are all distance $1$.  
By Theorem~1.2 and Corollary~1.3 of \cite{MPR}, if $P(\alpha,\X)$ is a two-bridge link so that $\Sigma(P_5(\alpha,\X)) = M_5(\tilde{\alpha},\tilde{\X})$ is a lens space, then $\tilde{\X}$ may be taken to be $\infty$ by the automorphisms of $M_5$.  Hence a triple of distance $1$, non-hyperbolic fillings of $M_5(\tilde{\alpha})$ may be taken (by an automorphism of $M_5$) to be the fillings $0, 1,\infty$.  
Descending back to $P_5$, these correspond to the fillings $\X = 0, \infty, -1$.  We will show that if $P_5(\alpha,0)$, $P_5(\alpha,\infty)$, and $P_5(\alpha,-1)$ are all two-bridge links, then the filling must simplify contrary to our assumption.

\begin{figure}
\centering
\psfrag{A}[Bc][Bc][.8]{\NW}
\psfrag{B}[Bc][Bc][.8]{\NE}
\psfrag{C}[Bc][Bc][.8]{\SW}
\psfrag{D}[Bc][Bc][.8]{\SE}
\psfrag{0}[Bc][Bc]{$L_0$}
\psfrag{i}[Bc][Bc]{$L_{\infty}$}
\psfrag{m}[Bc][Bc]{$L_{-1}$}
\includegraphics[height=1in]{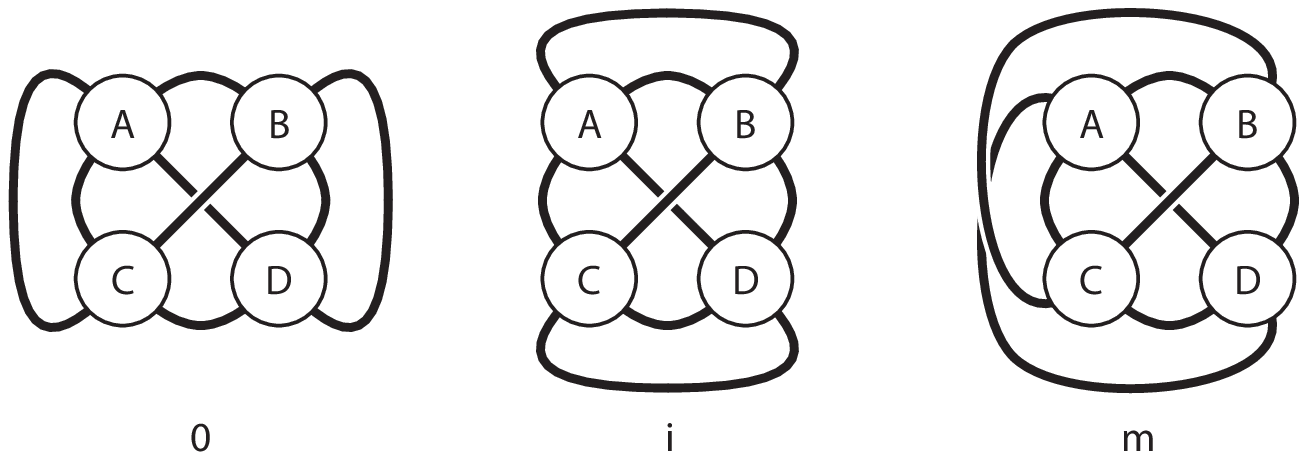}
\caption{ }
\label{fig:threelinks}
\end{figure}

The links $L_0=P_5(\alpha,0)$, $L_\infty=P_5(\alpha,\infty)$, and $L_{-1}=P_5(\alpha,-1)$ are shown in Figure~\ref{fig:threelinks}.  In each of these three links the there is an obvious Conway sphere (a sphere meeting the link in four points) that divides the link into pairs of sums of rational tangles.  View these spheres in Figure~\ref{fig:threelinks} as defined by a horizontal line, a vertical line, and the plane of the page (pull the rational tangles $\NE$ and $\SW$ in front of the page)  respectively.  We call the sums in these pairs North and South, West and East, and Front and Back accordingly.

For each of these links to be two-bridge, its Conway sphere must bound a rational tangle. Thus either North or South is rational, West or East is rational, and Front or Back is rational.  By the symmetries of $P_5$ that preserve $\X$ and its fillings, we may assume North and Front are always rational.

In order for the sum of rational tangles $A$ and $B$ (arranged as so: \tanglesum{a}{b} ) to be a rational tangle, at least one of the summands must be the reciprocal of an integer, i.e.\ a horizontal sequence of twists.  
Thus for West to be rational we must have either $\NW = [0,h]$ or $\SW = [0,k]$.  We could say ``West is rational by Northwest'' or ``West is rational by Southwest'' to indicate these two situations, but we will abbreviate these with $\WxNW$ and $\WxSW$.   Proceeding similarly with the other six cases we have the first two columns of the following chart which designate the case and the corresponding constraint on the filling $L_{\X}=P_5(\alpha,\X)$ that makes the Conway sphere bound a rational tangle. 

\[
\begin{array}{c||c|c|l}
\mbox{Case} & \mbox{Rational Constraint} & \mbox{Filling} & \mbox{Montesinos Link}\\
\hline\hline
\WxNW & \NW = [0,h]  &  \X=0      & L_0= Q([-1,h,\SW], [\NE], [\SE])\\
\WxSW & \SW =[0,k]    &  \X=0     & L_0=  Q([-1,k, \NW],[\SE], [\NE]) \\
\ExNE   & \NE = [0,m]  &   \X=0   & L_0= Q([-1,m,\SE],[\NW],[\SW]) \\
\ExSE   & \SE=[0,p]     & \X=0     &L_0  =Q([-1,p,\NE], [\SW],[\NW]) \\
\hline
\NxNW &   \NW=[n]     &  \X=\infty   & L_\infty= Q([1,n+\NE], [0,\SW], [0,\SE]) \\
\NxNE  &\NE=[\ell]       &  \X = \infty & L_\infty= Q([1,\ell+\NW], [0,\SE], [0,\SW]) \\
\hline
\FxNE &  \NW=[1,m]  & \X=-1   & L_{-1}=  Q([1,\NW], [m,1,\SW], [-1+\SE])\\
\FxSW  & \SW=[1,k]   & \X=-1  & L_{-1}= Q([1,\SE], [k,1,\NE], [-1+\NW])
\end{array}
\]

 When the Conway sphere of $L_{\X}=P_5(\alpha,\X)$ bounds a rational tangle, we can view $L_{\X}$ as a Montesinos link $Q(A,B,C)$.  The resulting Montesinos links for each of the cases are shown in the last column of the chart above.  This is illustrated for the cases $\NxNW$, $\WxNW$ and $\FxNE$ in Figure~\ref{fig:montesinostriple}; the other cases may be obtained through the symmetries of $P_5$.

\begin{figure}
\centering
\psfrag{A}[Bc][Bc][.8]{\NW}
\psfrag{B}[Bc][Bc][.8]{\NE}
\psfrag{C}[Bc][Bc][.8]{\SW}
\psfrag{D}[Bc][Bc][.8]{\SE}
\psfrag{0}[Bc][Bc]{$L_0$}
\psfrag{i}[Bc][Bc]{$L_{\infty}$}
\psfrag{m}[Bc][Bc]{$L_{-1}$}
\psfrag{n}[Bc][Bc]{\tiny$n$}
\psfrag{h}[Bc][Bc]{\tiny$h$}
\psfrag{s}[Bc][Bc]{\tiny$m$}
\includegraphics[height=2in]{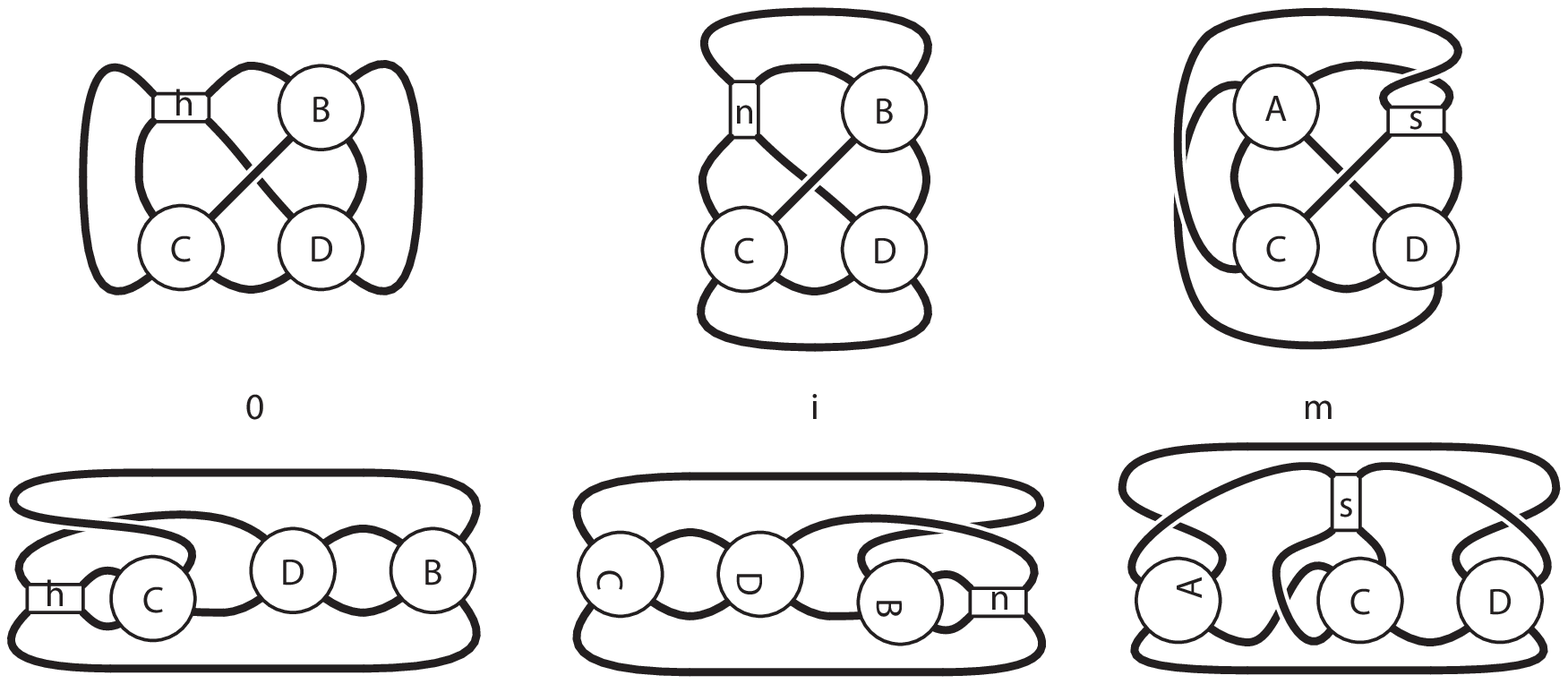}
\caption{ }
\label{fig:montesinostriple}
\end{figure}

For each $\X =0$, $\X=\infty$, and $\X=-1$ we actually want to determine when $L_{\X}$ is two-bridge and not just a Montesinos link. 
 Recall that for a Montesinos link $Q(A,B,C)$ to be two-bridge, at least one of its rational tangle factors $A$, $B$, $C$ must be associated to a number in $\hatQ$ that is the reciprocal of an integer.  Reciprocals of integers have continued fraction expansions of the form $[0,j]$ for $j\in\Z$.  The following two basic operations on continued fraction expansions facilitate the determination of conditions on a rational tangle factor to be of this form.
\begin{itemize}
\item $[a_1, \dots, a_{n-1}, a_n + [b_1, b_2, \dots, b_m]] = [a_1, \dots, a_{n-1}, a_n + b_1, b_2, \dots, b_m]$
\item If $r/s \in \hat{\Q}$ and $[a_1, \dots, a_n, r/s]=[0,j]$ for $a_1, \dots, a_n, j \in \Z$, then $r/s=[0,-a_n, \dots, -a_1, j]$.
\end{itemize}
 
The above chart  gives us $16$ cases for us to determine the conditions that force $L_0$, $L_\infty$, $L_{-1}$ to be two-bridge links.
\[
\begin{array}{cccc|cccc}
(1) & (2)  & (3) & (4) &  (5) & (6) & (7) &(8) \\
\WxNW &\WxNW &\WxNW &\WxNW &\WxSW &\WxSW &\WxSW &\WxSW \\
\NxNW & \NxNW & \NxNE& \NxNE &\NxNW & \NxNW &\NxNE& \NxNE\\
\FxNE &\FxSW &\FxNE &\FxSW &\FxNE &\FxSW &\FxNE &\FxSW \\
\hline
(9) & (10)  & (11) & (12) &  (13) & (14) & (15) &(16) \\
\ExNE &\ExNE &\ExNE &\ExNE &\ExSE &\ExSE &\ExSE &\ExSE \\
\NxNW & \NxNW & \NxNE& \NxNE &\NxNW & \NxNW &\NxNE& \NxNE\\
\FxNE &\FxSW &\FxNE &\FxSW &\FxNE &\FxSW &\FxNE &\FxSW 
\end{array}
\]
There are two order $3$ symmetries that relates some of these case.  There is an order $3$ rotation of $P_5$ that fixes the $\NE$ and $\X$ boundary components while cyclically permuting the triple $(\NW, \SW, \SE)$.  This causes a cyclic permutation of the triples $(\NxNW, \FxSW,\ExSE)$ and $(\NxNE, \FxNE,\ExNE)$. Under the action of this symmetry Cases (9), (12), (15) are equivalent  and Cases (10),(13),(16).

There is also an order $3$ rotation of $P_5$ that fixes the $\SE$ and $\X$ boundary components while cyclically permuting the triple $(\NW,\NE,\SW)$.  This causes a cyclic permutation of the triples $(\NxNW, \FxNE, \WxSW)$ and $(\WxNW, \NxNE, \FxSW)$.  Under the action of this symmetry Cases (1),(6),(7) are equivalent and Cases (2),(3),(8) are equivalent.

This leaves us with Cases (1),(2),(4),(5),(9),(10),(11),(14).  Each of these cases are treated in the following subsections. 
The general strategy is to first use the constraints from each specific case to find conditions for each of the three factors of each Montesinos link $L_0$, $L_\infty$, $L_{-1}$ that makes the link two-bridge.  We then discard situations that immediately imply (by Lemmas~\ref{lem:nonhypfilling} and \ref{lem:magictanglefilling}) that the filling simplifies.  We then proceed to consider conditions that force two or all three of these links to be two-bridge.  This frequently becomes an analysis of several subcases.  In all cases we eventually conclude that the filling simplifies.
\end{proof}

\subsubsection{Case (1): \WxNW, \NxNW, \FxNE}

This case gives the constraints $\NW=[0,h]$, $\NW=[n]$, and $\NE=[1,m]$.   The first two imply that $\NW=\pm1$.  The filling simplifies if $\NW=1$ so we take $\NW=-1$ (and so $h=1$ and $n=-1$).  Let us also set $\SW = r/s$ and $\SE = p/q$.  This gives us the three following Montesinos links.
\[
\begin{array}{c||lll}
\WxNW & L_0 &= Q([-1,1,r/s],[1,m],[p/q]) &=Q(-\tfrac{2s-r}{s-r},\tfrac{m-1}{m},\tfrac{p}{q}) \\
\NxNW & L_\infty &= Q([1,0,m], [0,r/s],[0,p/q]) &= Q(m+1, -\tfrac{s}{r}, -\tfrac{q}{p}) \\
\FxNE & L_{-1} &= Q([1,-1],[m,1,r/s],[-1+p/q]) &=Q(2,\tfrac{ms-mr-s}{s-r},\tfrac{p-q}{q})
\end{array}
\]

For each of these Montesinos links to be two-bridge, we need at least one numerator of their factors to be $\pm1$.  Hence at least one of the conditions in each of the following three rows must hold.

\[
\begin{array}{c||cc|c|c}
L_0 & m=0 & m=2  & r/s = [0,-1,1,j] & p/q = [0,j] \\
      & \NE = \infty & \NE = \tfrac12 & \SW = \tfrac{j-1}{2j-1} & \SE=-\tfrac1j \\
\hline
L_\infty & m=0 & m=-2& r/s = [i] & p/q =[i] \\
      & \NE = \infty & \NE = -\tfrac12 & \SW = i  & \SE = i \\
\hline
 L_{-1} &   &  & r/s =[0,-1,-m] & p/q = [1,h] \\
           &   &   & \SW = \tfrac{m}{m-1} & \SE = \tfrac{h-1}{h} 
 \end{array}
\]

Since $\NW = -1$, we see that the filling simplifies if any of the conditions on $\NE$ above hold.  Therefore either two of the $\SW$ conditions or two of the $\SE$ conditions above must hold.  

\smallskip

If $\SW = \tfrac{j-1}{2j-1} = i$ then $2j-1 = \pm1$ and so $j=0$ or $1$.  Hence $\SW = 1$ or $\SW =0$.

If $\SW = \tfrac{j-1}{2j-1} = \tfrac{m}{m-1}$ then $j-1=\epsilon m$ and $2j-1 = \epsilon (m - 1)$ for $\epsilon = \pm1$.  Thus $j=\mp1$ and so either $\SW = \tfrac23$ or $\SW = 0$.

If $\SW = i = \tfrac{m}{m-1}$ then $m=0$ or $m=2$.  Hence $\SW = 0$ or $\SW = 2$.

\smallskip

If $\SE = \tfrac{j-1}{3j-2} = i$ then $3j-2=\pm1$ and so $j=1$.  Hence $\SE = 0$.

If $\SE = \tfrac{j-1}{3j-2} = \tfrac{h-1}{h}$ then $j-1 = \epsilon (h -1)$ and $3j-2=\epsilon h$ for $\epsilon = \pm1$.  Thus  $\epsilon = 2j -1$ which implies $j=0$ or $j=1$.  Hence $\SW= \tfrac12$ or $\SW = 0$.

If $\SE = i = \tfrac{h-1}{h}$ then $h=\pm1$ Hence $\SE = 0$ or $\SE = 2$.

Since $\NW=-1$, these pairs of conditions all show the filling simplifies. \qed

\subsubsection{Case (2):  \WxNW, \NxNW, \FxSW}

This case gives the constraints $\NW=[0,h]=[n]$ and $\SW=[1,k]$.    The first pair of constraints implies $\NW = \pm1$.  Since the filling simplifies if $\NW=1$, we take $\NW=-1$ (and so $h=1$ and $n=-1$).  Let us also set $\NE= m/\ell$ and $\SE = p/q$.  This gives us the three following Montesinos links.
\[
\begin{array}{c|lll}
\WxNW & L_0 &=Q([-1,1,1,k], [m/\ell], [p/q]) &=Q(k-2, m/\ell, p/q) \\
\NxNW & L_\infty &= Q([1,-1+m/\ell], [0,1,k],[0,p/q]) &=Q(\tfrac{m-2\ell}{m-\ell},-\tfrac{k}{k-1},-q/p) \\
\FxSW & L_{-1} &=Q([1,p/q],[k,1,m/\ell],[-2]) &=Q(\tfrac{q-p}{q},\tfrac{k\ell-km-\ell}{\ell-m},-2)
\end{array}
\]

For each of these Montesinos links to be two-bridge, we need at least one of the numerators of their factors to be $\pm1$.  Hence at least one of the conditions in each of the following three rows must hold.

\[
\begin{array}{c||cc|c|c}
L_0 &  k=1 & k=3  & m/\ell =[0,j] & p/q = [0,j'] \\
        &  \SW=0 & \SW = \tfrac23 & \NE = -\tfrac1j  & \SE = -\tfrac1j\\
\hline
L_\infty & k=1 & k=-1 &     m/\ell = [1,-1,i]         &p/q = [i] \\
              & \SW=0 & \SW=2 & \NE = \tfrac{2i+1}{i+1} & \SE = i \\
\hline
L_{-1} &  &  &  m/\ell =[0,-1,-k,h]         & p/q = [1,h]\\
           &    &   & \NE = \tfrac{hk-2h+1}{hk-h+1} & \SE = \tfrac{h-1}{h}
\end{array}
\]

Bearing in mind that $\NW = -1$, the conditions on $k$ in the first two columns above imply the filling simplifies.  Therefore either two of the conditions on $\NE$ or two of the conditions on $\SE$ above must hold.  Note that the all three of the conditions on $\SE$ cannot hold.  We examine the six situations now.

\smallskip
If $\SE = -\tfrac1j=i$ then $\SE = 1$ or $\SE = -1$.  

If $\SE = -\tfrac1j=\tfrac{h-1}{h}$ then $h=0$ or $h=2$ and hence $\SE = \infty$ or $\SE=\tfrac12$.  

If $\SE = i = \tfrac{h-1}{h}$ then $h=1$ or $h=-1$ and hence $\SE = 0$ or $\SE = 2$.

\smallskip

If $\NE=-\tfrac1j = \tfrac{2i+1}{1+i}$ then $i=-1$ or $i=0$ and hence $\NE = \infty$ or $\NE=2$.

If $\NE = -\tfrac1j = \tfrac{hk-2h+1}{hk-h+1}$ then $h(k-2) = 0$ or $-2$.  Thus either $h=0$, $k=2$ or $(h,k) \in \{(1,0) (-1,4 ),$  $ (2,1), (-2,3)\}$.  Hence either $\NE=0$, $\SW=\tfrac12$, or $(\NE,\SW) \in\{(\infty, \infty), (\tfrac32,\tfrac34), (3,0),(\tfrac53,\tfrac23)\}$.  

If $\NE = \tfrac{2i+1}{1+i} = \tfrac{hk-2h+1}{hk-h+1}$ then $2i+1 = \epsilon(hk-2h+1)$ and $i+1 = \epsilon(hk-h+1)$ for $\epsilon = \pm1$.  Eliminating $i$, we obtain that $hk = \epsilon-1$.  Thus either $h=0$, $k=0$, or $(h,k) \in \{(1,-2), (-1, 2), (2, -1), (-2,1)\}$.  Hence either $\NE=0$, $\SW=\infty$, or $(\NE,\SW) \in \{(\tfrac12,\tfrac32), (\infty,\tfrac12),(\tfrac13,2), (-1,0)\}$.
\smallskip

Keeping in mind that $\NW=-1$,  it follows that every conclusion of the six situations implies the filling simplifies.

\subsubsection{Case (4): \WxNW, \NxNE, \FxSW }
This case gives the constraints $\NW=[0,h]$, $\NE=[\ell]$, and $\SW=[1,k]$.  Let us set $\SE =p/q$.  This gives us the three following Montesinos links.
\[
\begin{array}{c|lll}
\WxNW & L_0 &=Q([-1,h,1,k],[\ell],[p/q]) &=Q(-\tfrac{hk - h-1}{hk-h-k}, \ell, p/q) \\
\NxNE & L_\infty &=Q([1,\ell,h], [0,p/q], [0,1,k]) &=Q(\tfrac{h \ell -h-1}{h \ell -1}, -q/p, -\tfrac{k}{k-1}) \\
\FxSW & L_{-1} &=Q([1,p/q],[k,1,\ell],[-1,h]) &= Q(-\tfrac{p-q}{q}, \tfrac{k\ell - k -\ell}{\ell-1}, \tfrac{h+1}{h}) 
\end{array}
\]

For each of these Montesinos links to be two-bridge, we need at least one of the numerators of their factors to be $\pm1$.  Hence at least one of the conditions in each of the following three rows must hold.  (The first three of these conditions in each row are derived from the conditions $h(k-1)=1\pm1$, $h(\ell-1) = 1\pm1$, and $(k-1)(\ell-1) = 1\pm1$ respectively.)

\[
\begin{array}{c||ccc|cc|c}
L_0 & h=0 & k=1& (h,k) \in \{(1,3), (-1,-1), (2,2), (-2,0)\} & \ell = 1 & \ell = -1 & p/q = [0,j] \\
 & \NW=\infty & \SW = 0 & (\NW, \SW)\in \{(-1,\tfrac23) (1,2),(-\tfrac12,\tfrac12), (\tfrac12,\infty) \} &\NE =1 & \NE = -1 & \SE=-\tfrac1j \\
\hline
L_\infty & h=0&\ell=1& (h,\ell) \in \{(1,3), (-1,-1), (2,2), (-2,0)\}& k=1 & k=-1 & p/q = [i] \\
&\NW = \infty & \NE =1 & (\NW,\NE) \in \{(-1,3), (1,-1), (-\tfrac12,2),(\tfrac12,0)\}& \SW=0 & \SW=2 & \SE = i \\
\hline
L_{-1} & k=1 & \ell=1 & (k,\ell) \in\{(2,3),(3,2),(0,-1),(-1,0)\} & h=0 & h=-2 & p/q = [1,h]\\
& \SW = 0 & \NE=1 & (\SW, \NE) \in \{(\tfrac12,3),(\tfrac23,2), (\infty,-1)(2,0)\} & \NW = \infty & \NW=\tfrac12 & \SE=\tfrac{h-1}{h}
\end{array}
\]
Many of these conditions imply the filling simplifies.  We pare down these conditions to the following ones which do not alone imply that the filling simplifies.

\[
\begin{array}{c||c|cccc}
L_0 & \ell = -1 & p/q = [0,j] \\
 & \NE = -1 & \SE=-\tfrac1j \\
\hline
L_\infty &  k=-1 & p/q = [i] \\
& \SW=2 & \SE = i \\
\hline
L_{-1} & h=-2 & p/q = [1,h]\\
 & \NW=\tfrac12 & \SE=\tfrac{h-1}{h}
\end{array}
\]

Any two conditions of the middle column imply the filling simplifies.  Also all three conditions on $\SE$ cannot simultaneously hold.  Thus exactly two of the conditions on $\SE$ must hold and the other condition in the remaining row must hold.

If $\SE=i=-\tfrac1j$ then $\SE = 1$ or $\SE=-1$.   Since $\NW=\tfrac12$ by the third row,  the filling simplifies.

If $\SE = -\tfrac1j=\tfrac{h-1}{h}$ then $h=0$ or $h=2$ and so $\SE = \infty$ or $\SE = \tfrac12$.  Since $\SW=2$ by the second row, the filling simplifies.

If $\SE = i = \tfrac{h-1}{h}$ then $h=\pm1$ and so $\SE = 0$ or $\SE=2$.  Since $\NE=-1$ by the first row, the filling simplifies.\qed

\subsubsection{Case (5): \WxSW, \NxNW,\FxNE}

This case gives the constraints $\SW = [0,k]$, $\NW=[n]$, and $\NE=[1,m]$.  Let us set $\SE = p/q$.  This gives us the three following Montesinos links.
\[
\begin{array}{c|lll}
\WxSW & L_0 &=Q([-1,k,n],[p/q],[1,m]) &=Q(\tfrac{1-n-kn}{kn-1}, \tfrac{m-1}{m}, p/q)\\
\NxNW & L_\infty &=Q([1,n+1,m],[0,0,k],[0,p/q]) &= Q(\tfrac{mn-1}{mn+m-1}, k, -q/p)\\
\FxNE & L_{-1} &=Q([1,n],[m,1,0,k],[-1+p/q]) &=Q(\tfrac{n-1}{n}, \tfrac{mk+m-1}{k+1}, \tfrac{p-q}{q})
\end{array}
\]

For each of these Montesinos links to be two-bridge, we need at least one of the numerators of their factors to be $\pm1$.  Hence at least one of the conditions in each of the following three rows must hold. (The first set of three conditions in each row are derived from the conditions $n(k+1)=1\pm1$, $mn=1\pm1$, and $m(k+1) = 1 \pm1$ respectively.)

\[
\begin{array}{c||ccc|cc|c}
L_0 & n=0 & k= -1 & (n,k) \in \{(1,1), (-1,-3),(2,0), (-2,-2)\}& m=0 & m=2 & p/q = [0,j] \\
    & \NW = 0 & \SW =1 & (\NW, \SW) \in \{(1,-1), (-1, \tfrac13), (2, \infty), (-2, \tfrac12) \} & \NE = \infty & \NE=\tfrac12 & \SE = -\tfrac1j \\
\hline
L_\infty & m=0 & n=0  & (m,n) \in \{(1,2), (-1,-2), (2,1), (-2,-1)\}  & k=-1 &k=1  & p/q =[i] \\
& \NE = \infty & \NW = 0 & (\NE, \NW) \in \{(0,2), (2,-2), (\tfrac12,1),(\tfrac32,-1)\}&\SW=1 & \SW=-1 &\SE=i \\
\hline
L_{-1} &  m=0 & k=-1 & (m,k) \in \{(1,1), (-1,-3),(2,0), (-2,-2)\} & n=0 & n=2 & p/q = [1,h] \\
& \NE=\infty & \SW = 1 & (\NE, \SW) \in \{(0,-1), (2,\tfrac13 ), (\tfrac12, \infty ), (\tfrac32, \tfrac12)\} & \NW=0 & \NW=2 & \SE=\tfrac{h-1}{h}
\end{array}
\]
Many of these conditions imply the filling simplifies.  We pare down these conditions to the following ones which do not alone imply that the filling simplifies.

\[
\begin{array}{c||ccc|cc|c}
L_0 &  m=2 & p/q = [0,j] \\
    &   \NE=\tfrac12 & \SE = -\tfrac1j \\
\hline
L_\infty & k=1  & p/q =[i] \\
&  \SW=-1  &\SE=i \\
\hline
L_{-1} &  n=2 & p/q = [1,h] \\
&  \NW=2 & \SE=\tfrac{h-1}{h}
\end{array}
\]

Any two conditions of the middle column imply the filling simplifies.  Also all three conditions on $\SE$ cannot simultaneously hold.  Thus exactly two of the conditions on $\SE$ must hold and the other condition in the remaining row must hold. 

If $\SE=-\tfrac1j =i$ then $\SE = 1$ or $\SE=-1$.   Since $\NW=2$ by the third row,  the filling simplifies.

If $\SE = i = \tfrac{h-1}{h}$ then $h=\pm1$ and so $\SE = 0$ or $\SE=2$.  Since $\SW=-1$ by the second row, the filling simplifies.

If $\SE = -\tfrac1j=\tfrac{h-1}{h}$ then $h=0$ or $h=2$ and so $\SE = \infty$ or $\SE = \tfrac12$.  Since $\NE=\tfrac12$ by the first row, the filling simplifies.   \qed

\subsubsection{Case (9):   \ExNE, \NxNW, \FxNE}

This case gives the constraints $\NE = [0,m]$, $\NW=[n]$, and $\NE=[1,m']$.  The first and last imply that $\NE = \infty$ or $\tfrac12$.  The filling simplifies if $\NE=\infty$ so we take $\NE=\tfrac12$ (and so $m=-2$ and $m'=2$).  Let us also set $\SW = r/s$ and $\SE = p/q$. This gives us the three following Montesinos links.
\[
\begin{array}{c||lll}
\ExNE & L_0 &=Q([-1,-2,p/q],[n],[r/s]) &=Q(\tfrac{-3p+q}{2p-q},n,r/s)\\
\NxNW & L_\infty & = Q([1,n,-2],[0,r/s],[0,p/q]) &=Q(\tfrac{2n-1}{2n+1}, -s/r, -q/p)\\
\FxNE & L_{-1} &= Q([1,n],[-2,1,r/s],[-1+p/q])&=Q(\tfrac{n-1}{n}, \tfrac{r-2s}{r-s}, \tfrac{p-q}{q})
\end{array}
\]

For each of these Montesinos links to be two-bridge, we need at least one numerator of their factors to be $\pm1$.  Hence at least one of the conditions in each of the following three rows must hold.

\[
\begin{array}{c||cc|c|c}
L_0 & n=1 & n=-1 & r/s = [0,j] &  p/q =[0,2,-1,j] \\
        & \NW = 1 & \NW =-1 & \SW = -\tfrac1j & \SE =\tfrac{j-1}{3j-2} \\
\hline
L_\infty & n=0 & n=1 &   r/s = [i] & p/q = [i] \\
             & \NW = 0 & \NW=1 & \SW = i & \SE = i \\
\hline
L_{-1} & n=0 & n=2 &  r/s = [1,-1,h] & p/q = [1,h]\\
          & \NW = 0 & \NW =2 & \SW=\tfrac{2h+1}{h+1} & \SE=\tfrac{h-1}{h}
\end{array}
\]
Since  $\NE = \tfrac12$, we see that the filling simplifies if any of the $\NW$ conditions above hold.  Therefore either two of the $\SW$ conditions or two of the $\SE$ conditions above must hold.  

\smallskip
If $\SW=-\tfrac1j = i$ then $\SW = \pm1$.  

If $\SW =-\tfrac1j = \tfrac{2h+1}{h+1}$ then $2h+1 = \pm1$ and so $h=0$ or $-1$.  Hence $\SW = 1$ or $\SW=\infty$.

If $\SW = i = \tfrac{2h+1}{h+1}$ then $h+1 = \pm1$ and so $h=0$ or $-2$.  Hence $\SW = 1$ or $\SW=3$.
\smallskip

If $\SE=\tfrac{j-1}{3j-2} = i$ then $3j-2 = \pm1$ and so $j=1$.  Hence $\SE=0$.

If $\SE = \tfrac{j-1}{3j-2} = \tfrac{h-1}{h}$ then $\tfrac{-2j+1}{3j-2} = -\tfrac{1}{h}$.  Thus $-2j+1=\pm1$ so that $j=0,1$ and hence $\SE=2$ or $\SE=1$.

If $\SE = i = \tfrac{h-1}{h}$ then $h=\pm1$ and hence $\SE = 0$ or $\SE=-2$.

\smallskip
Since $\NE =\tfrac12$, these pairs of conditions all show that the filling simplifies. 

\qed

\subsubsection{Case (10):  \ExNE, \NxNW, \FxSW} 

This gives us the constraints $\NE = [0,m]$, $\NW=[n]$, and $\SW=[1,k]$.  Let us set $\SE=p/q$.   This gives us the three following Montesinos links.
\[
\begin{array}{c|lll}
\ExNE & L_0 &= Q([-1,m,p/q]),[n],[1,k] &=Q( -\tfrac{pm-q+p}{pm-q}, n, \tfrac{k-1}{k}) \\
\NxNW & L_\infty & =Q([1,n,m],[0,1,k],[0,p/q]) &=Q(\tfrac{mn-m-1}{mn-1}, -\tfrac{k}{k-1}, -q/p )\\
\FxSW & L_{-1} &= Q([1,p/q],[k,1,0,m],[n-1]) &=Q(\tfrac{p-q}{q}, \tfrac{km+k-1}{m+1}, n-1) 
\end{array}
\]

For each of these Montesinos links to be two-bridge, we need at least one numerator of their factors to be $\pm1$.  Hence  at least one of the conditions in each of the following three rows must hold.

\[
\begin{array}{c||c|cc|cc|cc|cc}
L_0  && n=-1 & n=1 & k=0 & k=2 & p/q = [0,-m,1,j] \\
        && \NW = -1 & \NW = 1& \SW = \infty & \SW=\tfrac12 & \SE = \tfrac{j-1}{mj+j-m} \\
\hline
L_\infty & \begin{array}{ccc} m=0&\quad &n=1 \end{array}& & &k=1 &k=-1 & p/q=[i] \\
           &  \begin{array}{ccc}\NE = \infty &\quad& \NW=1 \end{array} && & \SW=0 & \SW = 2 & \SE=i\\
 &(m,n) \in \{(1,3),(2,2), (-1,-1), (-2,0)\} &&&&& \\
& (\NE,\NW) \in \{(-1,3),(-\tfrac12,2), (1,-1),(\tfrac12,0) \}&&&&\\
\hline
L_{-1}  &\begin{array}{ccc} k=0 &\quad& m=-1 \end{array}& n=0 & n=2&&& p/q=[0,-1,h] \\
   & \begin{array}{ccc}\SW = \infty &\quad& \NE = 1 \end{array} & \NW = 0 & \NW = 2&&& \SE = \tfrac{h}{h+1}  \\
& (k,m) \in \{(1,1),(2,0),(-1,-3),(-2,-2)\}&&&&\\
 & (\SW,\NE) \in \{(0,-1),(\tfrac12,\infty),(2,\tfrac13),(\tfrac32,\tfrac12)\}&&&&\\
\end{array}
\]

Discarding the constraints that directly imply the filling simplifies, we are left with the following constraints.

\[
\begin{array}{c||c|c|c}
L_0 & n=-1 & k=2 & p/q=[0,-m,1,j] \\
       & \NW =-1 & \SW=\tfrac12 &  \SE = \tfrac{j-1}{mj+j-m}\\
\hline
L_\infty & &k=-1 &p/q = [i] \\
             & & \SW=2 & \SE = i \\
\hline
L_{-1} & n=2 && p/q = [0,-1,h] \\
          & \NW=2 &&\SE=\tfrac{h}{h+1} 
\end{array}
\]
Since the filling simplifies if $(\NW, \SW) \in \{(-1,2), (2,\tfrac12),(2,2)\}$, two of the constraints on $\SE$ must hold.   

If $\SE = \tfrac{j-1}{mj+j-m} = i$ then $j(m+1)-m = \pm1$ and hence $j=\tfrac{m\pm1}{m+1}$.  Thus either $m=0$ or $m=-2$.  If $m=0$ then $\NE=\infty$.  If $m=-2$ then $\NE=\tfrac12$ and either $j=1$ or $j=3$ and so $\SE = 0$ or $(\NE,\SE) =(\tfrac12,-2)$.  Hence the filling simplifies.

If $\SE = \tfrac{j-1}{mj+j-m} = \tfrac{h}{h+1}$ then $j-1 = \epsilon h$ and $mj+j-m = \epsilon(h+1)$ for $\epsilon = \pm1$.  Thus $m(j-1) = -1 +\epsilon = 0$ or $-2$.  Thus either $m=0$, $j=1$, or $(m,j) \in \{(1,3),(2,2), (-1,-1), (-2,0)\}$.   Therefore either $\NE=\infty$, $\SE=0$, or $(\NE,\SE) \in \{(\infty,\tfrac12),(2,1),(\tfrac12,-\tfrac12),(\tfrac23,-1)\}$.  Hence the filling simplifies.

If $\SE = i = \tfrac{h}{h+1}$ then $h=0$ or $h=-2$ and hence either $\SE=0$ (so that the filling simplifies) or $\SE =-2$.  Since these constraints on $\SE$ imply that $L_\infty$ and $L_{-1}$ are two-bridge, to have $L_0$ two-bridge we must have either $\NW=-1$, $\SW=\tfrac12$, or $\SE=\tfrac{j-1}{mj+j-m}$.  The last of these three puts us in the previous two cases.  The first two with $\SE=-2$ imply the filing simplifies.

\subsubsection{Case (11): \ExNE, \NxNE, \FxNE}

This case gives the constraints $\NE = [0,m] = [\ell] = [1,m']$.  But this has no solution for integers $\ell, m, m'$.

\subsubsection{Case (14): \ExSE,\NxNW,\FxSW}

This case gives the constraints $\SE=[0,p]$, $\NW=[n]$, and $\SW=[1,k]$.  Let us set $\NE=m/\ell$.   This gives us the three following Montesinos links.
\[
\begin{array}{c|lll}
\ExSE & L_0 &= Q([-1,p,m/\ell],[n],[1,k])&=Q(\tfrac{pm+m-\ell}{-pm+\ell},n, \tfrac{k-1}{k})\\
\NxNW & L_\infty &= Q([1,n+m/\ell],[0,1,k],[0,0,p])&=Q( \tfrac{n\ell+m -\ell}{n \ell +m},\tfrac{k}{-k+1}, p)\\
\FxSW & L_{-1} &= Q([1,0,p],[k,1,m/\ell],[n-1])&=Q(p+1, \tfrac{k\ell-km+m}{\ell-m},n-1)\\
\end{array}
\]

For each of these Montesinos links to be two-bridge, we need at least one numerator of their factors to be $\pm1$.  Hence  at least one of the conditions in each of the following three rows must hold.

\[
\begin{array}{c||cc|cc|cc|c}

L_0     & n=1& n=-1 & k=0& k=2 &        &       & \tfrac{m}{\ell}=[0,-p,1,j] \\
& \NW=1 & \NW=-1 & \SW=\infty&\SW=\tfrac12 & & & \NE= \tfrac{j-1}{j-p+jp} \\
\hline
L_\infty &         &   & k=1&k=-1 & p=1&p=-1  & \tfrac{m}{\ell}= [-n,-1,i] \\
    &         &      & \SW = 0& \SW=2 &\SE=-1 & \SE=1&\NE=\tfrac{i-n-ni }{1+i}\\
\hline
L_{-1} & n=0&n=2       &  &          & p=0&p=-2  & \tfrac{m}{\ell}= [0,-1,-k,h]\\
   & \NW=0 & \NW=2 &   &  &\SE =\infty & \SE =\tfrac12&\NE=\tfrac{1+kh}{1-h+kh}
\end{array}
\]

One observes that either the filling simplifies or at least one of the conditions in the following three rows must hold and in at least two rows the condition for $\NE$ holds.  (The cases with $\NW,\SW,\SE \in \{0,1,\infty\}$ from above are discarded below.  Then note that any pair of conditions on $\NW,\SW,\SE$ from two of the following rows implies the filling simplifies.)
\[
\begin{array}{c||c|c|c|c}
L_0      &  n=-1 & k=2 &               & \tfrac{m}{\ell}=[0,-p,1,j] \\
&  \NW=-1  &\SW=\tfrac12 & &\NE= \tfrac{j-1}{j-p+jp}  \\
\hline
L_\infty &            &k=-1 & p=1 & \tfrac{m}{\ell}= [-n,-1,i] \\
    &         & \SW=2 &\SE=-1 & \NE=\tfrac{i-n(1+i) }{1+i}\\
\hline
L_{-1}  &n=2         &          &p=-2  & \tfrac{m}{\ell}= [0,-1,-k,h]\\
   &  \NW=2 &     & \SE =\tfrac12 & \NE=\tfrac{1+kh}{1+kh-h}
\end{array}
\]

Finally we show that if two of the above equations for $\NE$ hold simultaneously, then the filling simplifies.

\smallskip
\noindent 
{\bf Case $0,\infty$:}  If $\tfrac{j-1}{j+p(j-1)}  = \tfrac{i-n(1+i) }{1+i}$ then  $j-1=\epsilon(i-n(1+i))$ and $j+p(j-1)=\epsilon(1+i)$ for $\epsilon =\pm1$.
Eliminating $j$ we obtain $i = \tfrac{\epsilon -1-n-np}{n-p+np}$. Thus $(n-1)(p+1)+1=\pm1$ and hence $(n-1)(p+1) = 0$ or $-2$.  Therefore either $n=1$, $p=-1$, or $(n,p) \in \{(2,-3), (3,-2),(0,1),(-1,0)\}$ which implies that either $\NW=1$, $\SE=1$, or $(\NW, \SE) \in \{(2,\tfrac13), (3,\tfrac12),(0,1),(-1,\infty)\}$.  These each imply the filling simplifies.

\smallskip
\noindent 
{\bf Case $0,-1$: }  If $\tfrac{j-1}{j+p(j-1)}  =\tfrac{1+kh}{1+kh-h}$ then $j-1 = \epsilon(1+kh)$ and $j+p(j-1)= \epsilon(1+kh-h)$ for $\epsilon = \pm1$.
Eliminating $j$ we obtain $h = -\tfrac{e+p}{1+kp}$.  Thus $1+kp=\pm1$ and hence $kp=0$ or $-2$.  Therefore either $k=0$, $p=0$, or $(k,p) \in \{(1,-2),(-1,2),(2,-1),(-2,1)\}$ which implies that either $\SW=\infty$, $\SE=\infty$, or $(\SW, \SE) \in \{(0,\tfrac12),(2,-\tfrac12),(\tfrac12,1), (\tfrac32,-1)\}$. These each imply the filling simplifies.

\smallskip
\noindent 
{\bf Case $\infty,-1$: }  If $ \tfrac{i-n(1+i) }{1+i}  =\tfrac{1+kh}{1+kh-h}$  then $i-n(1+i) = \epsilon(1+k h)$ and $1+i = \epsilon(1+kh-h)$.

Eliminating $i$ we obtain $h=-\tfrac{\epsilon +  n}{1 - n(k-1)}$.  Thus $ 1-n(k-1)=\pm1$ and hence $n(k-1) =0$ or $2$.  Therefore either $n=0$, $k=1$, or $(n,k) \in \{(1,3), (-1,-1), (2,2), (-2,0)\}$ which implies either $\NW=0$, $\SW=0$, or $(\NW,\SW) \in \{(1,\tfrac23), (-1,2),(2,\tfrac12),(-2,\infty)\}$.  These all imply the filling simplifies.
\qed

\subsection{
One-cusped hyperbolic fillings of the Magic Manifold with two or three lens space fillings.}

Recall that $M_3$, the magic manifold, is the exterior of the $(2,2,2)$--pretzel link.  For each boundary component of its presentation as the exterior of this pretzel link, use the standard meridian-longitude coordinates of the corresponding unknot component.   Martelli-Petronio classified all the non-hyperbolic fillings of $M_3$ \cite{MPMagic2006}, though there they used the mirror $N$ of $M_3$.  We use the present orientation to be consistent with Martelli-Petronio-Roukema \cite{MPR} (who acknowledge this difference in the sentence preceding their Theorem 2.3).  Also note that \cite{MPMagic2006} does not keep track of orientations of lens spaces.

\begin{figure}
\centering
\psfrag{n}[c][c]{\tiny $n$}
\psfrag{o}[c][c]{\tiny $-n$}
\psfrag{m}[c][c]{\tiny$m$}
\psfrag{p}[c][c]{\tiny$p/q$}
\psfrag{a}[c][c]{\small $P_3(n,4-n-\tfrac1m)$}
\psfrag{b}[c][c]{\small$P_3(3-\tfrac1m,\tfrac{p}{q})$}
\psfrag{c}[c][c]{\small$P_3(2-\tfrac1m,\tfrac{p}{q})$}
\psfrag{d}[c][c]{\small$P_3(1-\tfrac1m,1-\tfrac1n)$}
\includegraphics[width=5in]{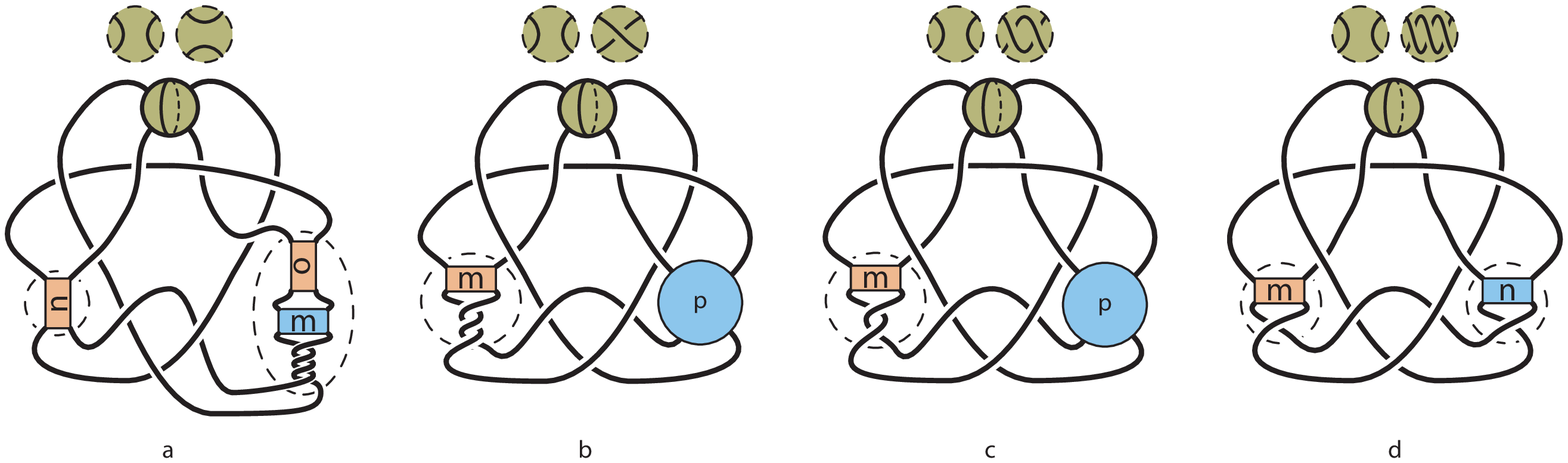}
\caption{}
\label{fig:magictanglelensfillings}
\end{figure}

\begin{theorem}\label{thm:MMtwolensfillings}
A one-cusped hyperbolic manifold obtained by filling $M_3$ that has two lens space fillings is homeomorphic to one of the following manifolds with the given filling slopes.  These are the double branched covers of the tangles shown in Figure~\ref{fig:magictanglelensfillings}.
\begin{enumerate}
\item[(0)] $X^0_{m,n}=M_3(n,4-n-\tfrac{1}{m})$, $n\in \Z-\{0,1,2,3\}$, $(m,n) \neq (-1, 4)$ or $ (-1,5)$, with lens space fillings $0, \infty$:
\[
\begin{array}{rcccc}
X^0_{m,n}(0) &=&M_3(0, n,4-n-\tfrac{1}{m})&=&  L(6m-1,2m-1)\\
 X^0_{m,n}(\infty)&=&
M_3(\infty, n,4-n-\tfrac{1}{m})&=& L(-n(1-m(4-n))-m,1-m(4-n)) 
\end{array}
\]

\item[(1)] $X^1_{m,\tfrac{p}{q}}=M_3(3-\tfrac{1}{m}, \tfrac{p}{q})$,  $m\in \Z-\{0,1\}$, $\tfrac{p}{q} \in \hatQ-\{0,1,2,3,\infty\}$, with lens space fillings $1, \infty$:
\[
\begin{array}{rcccc}
X^1_{m,\tfrac{p}{q}}(1) &=&M_3(1,3-\tfrac{1}{m}, \tfrac{p}{q})  &=& L(2m(p-3q)+p-q, m(p-3q)-q)\\
  X^1_{m,\tfrac{p}{q}}(\infty)&=& M_3(\infty, 3-\tfrac{1}{m}, \tfrac{p}{q})   &=&  L(-n(3p-q)+p,3p-q) 
\end{array}
\]

\item[(2)] $X^2_{m,\tfrac{p}{q}}=M_3(2 - \tfrac{1}{m}, \tfrac{p}{q})$, $m\in \Z-\{-1,0,1\}$, $\tfrac{p}{q} \in \hatQ-\{0,1,2,3,\infty\}$, with lens space fillings $2, \infty$:
\[
\begin{array}{rcccc}
X^2_{m,\tfrac{p}{q}}(2)&=&M_3(2, 2 - \tfrac{1}{m}, \tfrac{p}{q})&=& L(3m(p-2q)-2p+q, m(p-2q)-p+q)\\
  X^2_{m,\tfrac{p}{q}}(\infty) &=&M_3(\infty, 2 - \tfrac{1}{m}, \tfrac{p}{q})&=&L(-n(2p-q)+p,2p-q)   
\end{array}
\]

\item[(3)] $X^3_{m,n} = M_3(1-\tfrac{1}{m}, 1-\tfrac{1}{n})$, $m,n \in \Z-\{-1,0,1\}$, with lens space fillings $3, \infty$:
\[ 
\begin{array}{rcccc}
X^3_{m,n}(3) &=& M_3(3, 1-\tfrac{1}{m}, 1-\tfrac{1}{n}) &=& L((1+2m)(1+2n)-4, m(1+2n)-2) \\
  X^3_{m,n}(\infty) &=&M_3(\infty, 1-\tfrac{1}{m}, 1-\tfrac{1}{n}) &=&L(m+n-1,-1) 
\end{array}
\]

\end{enumerate}
\end{theorem}

\begin{remark}
Generically these hyperbolic manifolds have exactly $5$ non-hyperbolic fillings occurring on the slopes $\{0,1,2,3, \infty\}$.  Martelli-Petronio fully describe the cases when there are more than $5$ \cite{MPMagic2006}.
\end{remark}


\begin{theorem}\label{thm:MMthreelensfillings}
A one-cusped hyperbolic manifold obtained by filling $M_3$ that has three lens space fillings is homeomorphic to one of the following manifolds with the given filling slopes.
\begin{enumerate}
\item $A_{m,n}=M_3(2-\tfrac{1}{m}, 3-\tfrac{1}{n})$, $m \in \Z-\{-1,0,1\}$, $n \in \Z-\{0,1\}$, with lens space fillings $\{1, 2, \infty\}$:
\[
\begin{array}{rll}
A_{m,n}(1) &=M_3(1, 2-\tfrac{1}{m}, 3-\tfrac{1}{n}) &=L(2mn+m+2n-1, mn+m+n)  \\
 A_{m,n}(2) &=M_3(2, 2-\tfrac{1}{m}, 3-\tfrac{1}{n}) &=L(3mn-3m-5n+2, mn-m-2n+1)\\
 A_{m,n}(\infty)&=M_3(\infty, 2-\tfrac{1}{m}, 3-\tfrac{1}{n}) &=L(5mn-2m-3n+1,3-5m)
\end{array}
 \]
\item $B_{\tfrac{p}{q}}=M_3(\tfrac{5}{2}, \tfrac{p}{q})$,  $\tfrac{p}{q} \in \hatQ -\{0,1,3/2,2,3,\infty\}$, with lens space fillings $1, 2, \infty$:
\[
\begin{array}{rll}
B_{\tfrac{p}{q}}(1) &=M_3(1,\tfrac{5}{2}, \tfrac{p}{q}) &=L(-3p+11q,2p-7q) \\ 
B_{\tfrac{p}{q}}(2) &=M_3(2, \tfrac{5}{2}, \tfrac{p}{q}) &=L(8p-13q,3p-5q) \\
 B_{\tfrac{p}{q}}(\infty) &=M_3(\infty, \tfrac{5}{2}, \tfrac{p}{q}) &=L(5p-2q,2p-q) 
\end{array}
 \]
\end{enumerate}
\end{theorem}

\begin{remark}
As stated in the introduction, we define the two families of manifolds 
\[ \A = \{A_{m,n} | m,n\in\Z\} \quad \mbox{and} \quad \B = \{B_{p/q} =\{B_{p/q} | p/q \in \hatQ\} \]
without the constraints that ensure their hyperbolicity.  These manifolds with their triples of lens space fillings may be viewed as the double branched covers of the tangles given in Figures~\ref{fig:FamilyA} and \ref{fig:FamilyB}.
\end{remark}

\begin{remark}
Further note the relationships from \cite{MPMagic2006} of these families with the Whitehead Sister Manifold and the Berge Manifold.
\begin{enumerate}
\item $A_{m,-1}$ is a filling of $M_3(4)$ and $A_{2,n}$ is a filling of $M_3(\tfrac{3}{2})$.   The manifolds $M_3(4)$ and $M_3(\tfrac{3}{2})$ are mirrors and are homeomorphic to the Whitehead Sister Manifold.

\item $B_{\tfrac{p}{q}}$ is a filling of the Berge Manifold $M_3(\tfrac{5}{2})$ and $M_3(\tfrac{5}{2},1), M_3(\tfrac{5}{2},2), M_3(\tfrac{5}{2},\infty)$ are all solid tori.
\end{enumerate}
\end{remark}

\begin{proof}[Proof of Theorem~\ref{thm:MMtwolensfillings}]
This follows from synthesizing the classification of non-hyperbolic fillings of the Magic Manifold \cite{MPMagic2006}.  Our $M_3$ has the opposite orientation as the Magic Manifold $N$ as it appears in that article.  This will cause an over-all sign change for the slopes.   For this proof we will use $N$ to ease reference to \cite{MPMagic2006} and then switch to $M_3$ at the end.

By \cite[Theorem 1.1]{MPMagic2006},  $N(\alpha)$ is hyperbolic unless $\alpha \in \{ \infty, -3, -2, -1, 0\}$.  By \cite[Theorem 1.2]{MPMagic2006}, $N(\alpha,\beta)$ is hyperbolic unless $\{\alpha, \beta\} \in \{ \{1, 1\}, \{-4, -\tfrac{1}{2}\}, \{-\tfrac{3}{2}, -\tfrac{5}{2}\} \}$ or either $\alpha$ or $\beta \in \{\infty, -3, -2, -1, 0\}$. Lastly, by \cite[Theorem 1.1]{MPMagic2006}, the closed manifold $N(\alpha, \beta, \gamma)$ is hyperbolic unless one slope is $\infty$ (and hence is a lens space) or it appears in Tables 2, 3, or 4 of \cite{MPMagic2006} with $\{\alpha, \beta, \gamma\} = \{\tfrac{p}{q}, \tfrac{r}{s}, \tfrac{t}{u}\}$.  In these tables the non-hyperbolic manifold produced is also described.  With the exception of $\RP^3$ and connect sums of lens spaces where one summand is actually $S^3$, lens spaces are described in the common form $L(a,b)$ for coprime $a,b$; the parameters are partitioned whereever small Seifert fibered spaces or graph manifolds degenerate to lens spaces or connect sums.

We may thus determine the hyperbolic manifolds $N(\alpha, \beta)$ with two lens space fillings by listing the lens spaces that appear in Tables 2, 3, and 4 of \cite{MPMagic2006}, and then discarding any for which their surgery description as $N(\alpha, \beta, \gamma)$ has either two of the slopes $\alpha, \beta, \gamma$ in $\{\infty, -3, -2, -1, 0\}$ or one slope in $\{\infty, -3, -2, -1, 0\}$ with the other two in $ \{ \{1, 1\}, \{-4, -\tfrac{1}{2}\}, \{-\tfrac{3}{2}, -\tfrac{5}{2}\} \}$. This ensures that for each lens space $N(\alpha, \beta, \gamma)$ remaining in the list, one of the surgeries may be drilled to result in a hyperbolic manifold.  This leaves us with the following list of hyperbolic manifolds $N(\alpha,\beta)$ with (at least) two lens space fillings:
\begin{enumerate}
\item $N(-1+\tfrac{1}{m}, -1+\tfrac{1}{n})$, $m,n \in \Z-\{-1,0,1\}$, with lens space fillings $-3, \infty$;
\begin{align*}
 N(-3; -1+\tfrac{1}{m}, -1+\tfrac{1}{n})&=L((2n+1)(2m+1)-4,(2n+1)m-2)  \\
 N(\infty; -1+\tfrac{1}{m}, -1+\tfrac{1}{n})&=L(m+n-1, 1)
\end{align*}

\item $N(-2 + \tfrac{1}{m}, \tfrac{t}{u})$, $m\in \Z-\{-1,0,1\}$, $\tfrac{t}{u} \in \hatQ-\{-3,-2,-1,0,\infty\}$, with lens space fillings $-2, \infty$;
\begin{align*}
N(-2; -2 + \tfrac{1}{m}, \tfrac{t}{u})&= L(-3n(t+2u)+2t+u,n(t+2u)-t-u)\\
N(\infty; -2 + \tfrac{1}{m}, \tfrac{t}{u})&=L(m(2t+u)-t, 2t+u)
\end{align*}

\item $N(-3+\tfrac{1}{m}, \tfrac{t}{u})$,  $m\in \Z-\{-1,0,1\}$, $\tfrac{t}{u} \in \hatQ-\{-3,-2,-1,0,\infty\}$, with lens space fillings $-1, \infty$;
\begin{align*}
 N(-1; -3+\tfrac{1}{m}, \tfrac{t}{u})&= L(-2m(t+3u)-t-u,m(t+3u)+u)   \\
 N(\infty; -3+\tfrac{1}{m}, \tfrac{t}{u})&= L( m(3t+u)-t,3t+u )
\end{align*}

\item $N(n,-4-n+\tfrac{1}{m})$, $n\in \Z-\{-3,-2,-1,0\}$, $(m,n) \neq (-1, -4)$ or $ (-1,-5)$, with lens space fillings $0, \infty$.
\begin{align*}
N(0; n,-4-n+\tfrac{1}{m}) &= L(6m-1, -2m+1) \\
N(\infty; n,-4-n+\tfrac{1}{m}) &= L(n(1-m(4+n))-m, 1-m(4+n))
\end{align*}
\end{enumerate}

These four all appear in Table 2 of \cite{MPMagic2006}, though for $N(0; n,-4-n+\tfrac{1}{m})$, $N(-2; -2+\tfrac{1}{n},\tfrac{t}{u})$, and $N(-3;-1+\tfrac{1}{m},-1+\tfrac{1}{m})$ we obtained the mirror of what is listed in Table 2. (Recall that \cite{MPMagic2006} expressly does not keep track of orientations of lens spaces.)   The lens space fillings in Table 3 are discarded.  There are no lens space fillings in Table 4.

With a flip of signs of surgery slopes to switch to $M_3$ and then a reparametrization, we have the stated result.
\end{proof}

\begin{proof}[Proof of Theorem~\ref{thm:MMthreelensfillings}]
By the Cyclic Surgery Theorem \cite{CGLS} we only need to determine constraints on the parameters of pairs of the fillings given in Theorem~\ref{thm:MMtwolensfillings} whose lens space filling slopes contain adjacent integers.  Since the only orientation preserving automorphism of $M_3$ that preserves one cusp and its slopes exchanges the other two, we have three cases.

\medskip
{\bf Case 1:}   $X^0_{m,n}=X^1_{m',{p'/q'}}$

Thus $M_3(n,4-n+\tfrac{1}{m}) = M_3(3+\tfrac{1}{m'}, \tfrac{p'}{q'})$.  

\smallskip
{\bf Case 1a:}  $n = 3+\tfrac{1}{m'}$ and $4-n+\tfrac{1}{m} = \tfrac{p'}{q'}$

Since $n$ and $m'$ are integers, $m' = \pm1$.  For this manifold to be hyperbolic we cannot have $m'=-1$.  Thus $m'=1$, $n = 4$ and $\tfrac{p'}{q'} = \tfrac{1}{m}$ giving the manifold $M_3(4,\tfrac{1}{m})$ with lens space fillings $0, 1, \infty$.

\smallskip
{\bf Case 1b:} $n  = \tfrac{p'}{q'}$ and $4-n+\tfrac{1}{m} = 3+\tfrac{1}{m'}$

The second equation yields $n-1 = \tfrac{1}{m} - \tfrac{1}{m'}$.  For the right side to be integral, we must have $m, m' \in \{-2,+2\}$ or $m,m' \in \{-1,+1\}$.  Every one of these choices results in $n \in \{0,1,2,3\}$ except $m=-1, m'=+1$ for which $n=-1$.  This gives the manifold $M_3(-1, 4)$ which is included among the manifolds of Case 1a.

\medskip
{\bf Case 2:} $X^1_{m',\tfrac{p'}{q'}}= X^2_{m'',\tfrac{p''}{q''}} $

 $M_3(3+\tfrac{1}{m'}, \tfrac{p'}{q'}) = M_3(2 + \tfrac{1}{m''}, \tfrac{p''}{q''})$

\smallskip
{\bf Case 2a:} $3+\tfrac{1}{m'} =2 + \tfrac{1}{m''}$     and $\tfrac{p'}{q'}=\tfrac{p''}{q''}$

The first equation gives $1 = -\tfrac{1}{m'}+\tfrac{1}{m''}$ from which we conclude $m'=-2$ and $m''=2$.  This gives the manifold $M_3(\tfrac{5}{2}, \tfrac{p}{q})$ with lens space fillings $1,2,\infty$.

\smallskip
{\bf Case 2b:}  $3+\tfrac{1}{m'}  =\tfrac{p''}{q''}$    and $\tfrac{p'}{q'}=2 + \tfrac{1}{m''}$

This gives the manifold $M_3(3+\tfrac{1}{m}, 2+\tfrac{1}{n})$ with lens space fillings $1,2,\infty$.

\medskip
{\bf Case 3:} $X^2_{m,\tfrac{p}{q}} = X^3_{m,n}$

$M_3(2 + \tfrac{1}{m''}, \tfrac{p''}{q''}) =  M_3(1+\tfrac{1}{m'''}, 1+\tfrac{1}{n'''})$

\smallskip
{\bf Case 3a:}   $ 2 + \tfrac{1}{m''}=   1+\tfrac{1}{m'''}     $ and   $\tfrac{p''}{q''}=1+\tfrac{1}{n'''}        $

The first equation gives $1=\tfrac{1}{m'''} -\tfrac{1}{m''}$ from which we conclude $m'''= 2$ and $m''=-2$.  This gives the manifold $M_3(\tfrac{3}{2}, 1+\tfrac{1}{m})$ with lens space fillings $2, 3, \infty$.

\smallskip
{\bf Case 3b:}   $2 + \tfrac{1}{m''}=  1+\tfrac{1}{n'''}       $ and   $\tfrac{p''}{q''}=  1+\tfrac{1}{m'''}      $

This is equivalent to Case 3a.

\medskip

Finally, we note that the manifold $M_3(3+\tfrac{1}{m}, 2+\tfrac{1}{n})$ with lens space fillings $\{1, 2, \infty\}$ of Case 2b subsumes Cases 1a and  3a.  
In both situations we apply the orientation reversing homeomorphism of   \cite[Proposition~1.5(1.3)]{MPMagic2006} which equates $M_3(\tfrac{3}{2},\alpha, \beta)$ with $M_3(4,\tfrac{1-\alpha}{2-\alpha}, 3-\beta)$.
 Setting   $n=-1$, $\beta=3+\tfrac{1}{m}$, and $\alpha = 1, 2$, or $\infty$, this homeomorphism takes Case 2b to $M_3(4, -\tfrac{1}{m})$ with lens space fillings $0, \infty, 1$ respectively.   Then send $m$ to $-m$ to obtain Case 1a. 
Alternatively, setting $m=1$, $\beta =2+\tfrac{1}{n}$, and $\alpha = 1, 2$, or $\infty$, the inverse of this homeomorphism takes Case 2b to $M_3(\tfrac{3}{2}, 1-\tfrac{1}{m})$ with lens space fillings $\infty, 3, 2$ respectively. Then sending $m$ to $n$ produces Case 3a.
\end{proof}

\section{Alternative surgeries on GOFK knots}\label{sec:altsurg}
Families VII and VIII of Berge's doubly primitive knots in $S^3$ are comprised of the knots that embed in the fiber of a trefoil or the figure eight knot respectively \cite{TheBergeResult}.  Collectively these knots may be regarded as those that embed in the fiber of a genus one fibered knot in $S^3$, and we refer to them as GOFK knots for short.

Given a doubly primitive knot $K$ with framing of slope $p$, let us say a lens space surgery on $K$ of slope $p+1$ or $p-1$ is an {\em alternative} surgery.  (Note that a torus knot does not have an alternative surgery.)
As mentioned in the introduction, it is conceivable that an alternative surgery does not come from a doubly primitive framing.  The goal of this section is to show that any alternative surgery on a GOFK knot indeed arises from a doubly primitive framing.

\begin{theorem}\label{thm:altGOFKsurg}
Let $K$ be a non-trivial GOFK knot with an alternative lens space surgery.  Then $K$ is either the pretzel knot $P(-2,3,7)$ or the knot obtained as $-1$--surgery on the unknotted component of the Whitehead Sister Link.
\end{theorem}

\subsection{A few technical tools}
Our proof of Theorem~\ref{thm:altGOFKsurg} requires a few technical tools which we collect here.

\subsubsection{Simple knots}
The surgery duals to doubly primitive knots in $S^3$ belong to a special class of $(1,1)$--knots \cite{TheBergeResult}.

If $K$ is a $(1,1)$--knot and there are meridional disks of the Heegaard solid tori $V_\alpha$ and $V_\beta$ disjoint from $K$ whose boundaries intersect minimally in the torus $V_\alpha \cap V_\beta$, then we say $K$ is a {\em simple knot}.  One may show that for each (torsion) first homology class in $L(p,q)$ there is a unique oriented simple knot $K(p,q,k)$ representing $k\mu$  where $\mu$ is the homology class of the core of, say, $V_\beta$ with some choice of orientation and $q \mu$ is the homology class of the core of $V_\alpha$.  (Note that the trivial knots are the only simple knots in $S^3$ and $S^1 \times S^2$.)  We will not be concerned with orientations on our simple knots.  We say two simple knots $K(p,q,k)$ and $K(p',q',k')$ are {\em equivalent} if there is an orientation preserving homeomorphism from $L(p,q)$ to $L(p',q')$ that takes $K(p,q,k)$ to $K(p',q',\pm k')$.  The following lemma is fairly straightforward.
\begin{lemma}[Lemma 2.5 \cite{Rasmussen}]\label{lem:simpleknotequiv}
Assume $p,p'>0$.  The simple knots $K(p,q,k)$ and $K(p',q',k')$ are equivalent if $p=p'$ and either $q=q' \mod p$ and $\pm k = k' \mod p$ or $q^{-1} = q' \mod p$ and $\pm q' k = k' \mod p$. \qed
\end{lemma}

Simple knots are also {\em Floer simple}, meaning that their knot Floer homology is as simple as possible \cite{HeddenSimple, Rasmussen}.  This is manifested by the fact that simple knots admit a ``grid number one'' presentation for which the knot Floer homology chain complex has no differentials (see e.g.\ \cite{BakerGrigsbyHedden} or the Example at the end of \cite[Section 3.5]{Rasmussen}).   This permits the easy calculation of the Euler characteristic of a simple knot by using Theorem~1.1 of \cite{ni2009link}.  
(The Euler characteristic of a knot --- null homologous or not --- is the maximal Euler characteristic among its generalized Seifert surfaces: the properly embedded, connected, oriented surfaces with coherently oriented boundary in the exterior of the knot.)
  This is well known to the experts but we have not found it written out explicitly in any references.  Since it will be of use to us in our proof of Theorem~\ref{thm:altGOFKsurg}, we record it here and sketch its proof.

\begin{theorem}\label{thm:simpleknoteulerchar}
The Euler characteristic of the simple knot $K(p,q,k)$ is
\[ \chi(K(p,q,k)) =   \tfrac{p}{\gcd(p,k)} \cdot (1-2 \max \mathcal{A}_{p,q,k})\]
where 
\begin{itemize}
\item $\mathcal{A}_{p,q,k}$ is the set  $\{A_i \}_{i=0}^{p-1}$ symmetrized about $0$,
\item the rational numbers $A_i$ are relatively defined by $A_i - A_{i+1} = \tfrac{1}{p} (\overline{i q^{-1}} - \overline{(i+k)q^{-1}})$, and 
\item $\overline{n}$ denotes the residue of $n \mod p$ in the set $\{0, 1, \dots, p-1\}$.
\end{itemize}
\end{theorem}

\begin{proof}[Sketch of Proof]
The grid number one diagram of $K(p,q,k)$ gives a knot Floer homology chain complex with $p$ generators and no differentials.  The numbers $\tfrac{p}{\gcd(p,k)} A_i$ are the relative Alexander gradings of these generators.  Symmetrizing the set of these gradings about $0$ makes them the absolute Alexander gradings (which may alternatively be computed by way of their Maslov gradings).  Note that $\tfrac{p}{\gcd(p,k)}$ is the homological order of the knot and equivalent to the minimal number of times a meridian of the knot intersects a generalized Seifert surface for the knot.  We then apply Theorem~1.1 of \cite{ni2009link} to obtain the Euler characteristic.   In our situation the term $y(h)$ of that theorem is $\tfrac{p}{\gcd(p,k)} \max \mathcal{A}_{p,q,k}$.  
\end{proof}

\subsubsection{Simple knots dual to GOFK knots}
From Greene's presentation \cite{Greene2013Realization} of Rasmussen's tabulation \cite{Rasmussen} of Berge's doubly primitive knots \cite{BergeSolidTori}: a doubly primitive knot in $S^3$ with framing of positive slope $p$ in the GOFK family VII or VIII is surgery dual to the simple knot $K(p,q,k)$ where 
\[  k^2+\epsilon(k+1) = 0 \mod p \tag{$\star$}\]
has a solution for some integer $0<k<p$ and $q \equiv -k^2 \mod p$. Here $\epsilon = +1$ or $-1$ for family VII or VIII respectively.   Note that any integer congruent to $k \mod p$ works equivalently.  We will use equation ($\star$) to determine if a lens space contains a knot dual to a GOFK knot and, if so, its homology class and hence the particular simple knot.

\subsubsection{Exponent sums of genus one fibered knots}
The {\em exponent sum}, an invariant of once-punctured torus bundles and genus one fibered knots (and $3$--braids)  helps to determine when two lens spaces may contain genus one fibered knots with monodromies that differ by a Dehn twist.

Let $T$ be the oriented torus minus an open disk, and let $a,b$ be two oriented simple closed curves in the interior of $T$ intersecting once transversally so that $a \cdot b =+1$.   For a curve $c$ in the interior of $T$, let $\tau_c$ be a positive Dehn twist along $c$.  The mapping class group $\mathcal{M}(T)$, the group of isotopy classes of orientation preserving diffeomorphisms of $T$ that act as identity on $\bdry T$, is well known to be isomorphic to the three strand braid group $B_3$.  Indeed,
\[ \mathcal{M}(T) = \langle \tau_a, \tau_b \vert \tau_a \tau_b \tau_a = \tau_b \tau_a \tau_b \rangle. \]
Given $\phi \in \mathcal{M}(T)$, we may therefore write $\phi$ as a word in $\tau_a$ and $\tau_b$ (with positive and negative powers).  From the presentation above, any such word for $\phi$ has the same exponent sum.  Define $E(\phi)$ to be this exponent sum.  Observe that $E(\phi)$ is actually an invariant of the conjugacy class of $\phi$.  
Thus for a genus one fibered knot $J$ with monodromy $\phi$ --- that is, a knot whose exterior is the once-punctured torus bundle $T \times [0,1] /(x,1)\sim(\phi(x),0)$ --- further define $E(J) = E(\phi)$.  
Indeed, viewing a genus one fibered knot as the lift of the braid axis in the double branched cover of a closed $3$--braid, this exponent sum agrees with the standard exponent sum for braids.

For $\epsilon = \pm1$, let $J_\epsilon$ be the genus one fibered knot in $S^3$ with monodromy $\tau^{\epsilon}_b  \tau_a$.  Then $J_+$ is the positive trefoil and $J_-$ is the figure eight knot.  Note that $E(J_\epsilon) = 1+\epsilon$.

\subsubsection{Constraints on alternative surgeries}

\begin{lemma}\label{lem:gofktool}
Assume $K \subset S^3$ is a non-trivial doubly primitive GOFK knot of slope $p>0$ in family VII or VIII.
If $p\pm1$--surgery on $K$ is a lens space $L(p\pm1,q)$  then
\begin{enumerate}[(a)]
\item $p\pm1$ is even and at least $18$,
\item $q \neq 1 \mod p$,
\item $L(p\pm1,q)$ contains a (nullhomologous) genus one fibered knot $J'$ with $E(J') \in \{-1,1,3\}$,
\item there is a doubly primitive knot $K' \subset S^3$ of slope $p\pm1$ in a Berge-Gabai family (I,II, III, IV, V) or a sporadic family (IX,X) such that $p\pm1$ surgery on $K'$ is $L(p\pm1,q)$, and
\item $\widehat{HFK}(K') \cong \widehat{HFK}(K)$ --- in particular $K$ and $K'$ have the same genus and Euler characteristic.
\end{enumerate}
Here $\pm1$ denotes a consistent choice of either $+1$ or $-1$.
\end{lemma}

\begin{proof}
(a) Berge shows that $p$ must be odd for knots in families VII and VIII \cite{BergeSolidTori}.  The first author shows that any knot in $S^3$ with an odd lens space surgery of order less than $19$ must be a torus knot \cite{Baker2008GOFK}.  Non-trivial torus knots do not have alternative surgeries.

(b) Positive surgery on a knot gives a lens space $L(n,1)$ only if the knot is trivial \cite{kmos}.

(c) Since $K$ is a knot in the fiber of a genus one fibered knot $J_\epsilon$ in $S^3$ and the fiber gives $K$ a framing of slope $p$, $p\pm1$ surgery effects a $\mp1$ Dehn twist along $K$ in the monodromy of $J_\epsilon$.   Thus $J_\epsilon$ is sent to a genus one fibered knot $J'$. In particular there exists $\zeta \in \mathcal{M}(T)$ such that the monodromy of $J'$ is conjugate to $(\tau_b^\epsilon \tau_a) (\zeta^{-1} \tau_a^{\mp1} \zeta)$.  Hence $E(J') = 1+\epsilon \mp 1$.

(d) Greene shows that any lens space obtained by a positive integral surgery on a knot in $S^3$ must be obtained by that surgery on one of Berge's doubly primitive knots along its doubly primitive slope \cite{Greene2013Realization}.  Due to (a), this means $K'$ must belong to a Berge-Gabai family or a sporadic family.  (As Rasmussen notes \cite{Rasmussen}, families VI, XI, and XII of Berge's original list of doubly primitive knots may be absorbed into the other families.)

(e) Greene also shows that the surgery duals to $K$ and $K'$ in $L(p\pm1,q)$ must be homologous (after reorienting one if needed) and the hat version of their knot Floer homologies must be isomorphic \cite{Greene2013Realization}.  Hence they have the same genus \cite{ni2009link,ghiggini2008knot}.
\end{proof}

\subsection{Norm Sequences for Lens Spaces}

The lens space $L(p,q)$ is obtained by $-p/q$ surgery on the unknot.
Whenever we have a continued fraction expansion $p/q = [a_1, a_2, \dots, a_n]$ with integer coefficients
the lens space $L(p,q)$ may also be obtained surgery on the linear chain link of $n$ components where the $i$th component has surgery coefficient $-a_i$.
To this continued fraction and surgery descriptions we correspond the integer sequence $(a_1, \dots, a_n)$.   If $a_i \geq 2$ for all $i$, we say $(a_1, \dots, a_n)$ is a {\em norm sequence} for the lens space $L(p,q)$.      If $a_i \geq 0$ for all $i$ or the sequence is the empty sequence $()$, then we say it is a {\em weak norm sequence} for the lens space.  (Our terminology is chosen to be similar to what appears in \cite{Greene2013Realization}.  There Greene obtains a vertex basis $\{x_1, \dots, x_n\}$ for a linear lattice $\Lambda(p,q)$ where the associated sequence of norms $\nu=(|x_1|, \dots, |x_n|)$ gives the coefficients for a continued fraction expansion of $p/q$.)   Observe that, through the surgery description correspondence, a weak norm sequence and its reverse may be regarded as equivalent.  Also a non-empty weak norm sequence in which some $a_i = 0$ or $1$ may be reduced to a shorter such sequence except for the sequences $(0)$ and $(1)$.  The sequence $(0)$ corresponds to the lens space $S^1 \times S^2$, and $(1)$ corresponds to $S^3$ as does $() = (a_1, 0)$. Further note that if $(a_1, \dots, a_n)$ is a norm sequence for a lens space then its reverse $(a_n, \dots, a_1)$ is the only other positive sequence of the same lens space.   We consider norm sequences up to this reversal.   

A word on notation. As does Greene, we use Lisca's convention that $2^{[t]}$ stands for the number $2$ repeated $t$ times in a sequence; e.g.\ $(\dots, 2^{[3]}, \dots) = (\dots, 2,2,2, \dots)$.  This should be clear when $t>0$.  The cases $t=0$ and $t=-1$ will also arise.  A $2^{[0]}$ in a sequence may simply be omitted.  Hence $(\dots, a, 2^{[0]}, b, \dots) = (\dots, a, b, \dots)$ and $(\dots, a, 2^{[0]}) = (\dots, a)$.   For $2^{[-1]}$ use the relations $(\dots, a, 2^{[-1]}, b,\dots) = (\dots, a+b-2, \dots)$ and $(\dots, a, b, 2^{[-1]}) = (\dots, a)$.  (One may understand these relations through the associated continued fractions or paths in the Farey Tesselation, for example.)

\begin{lemma}\label{lem:gofksurgerysequences}
Assume $L(p,q)$ contains a genus one fibered knot.  Then either $L(p,q) \cong S^3$ or $S^1 \times S^2$ or $L(p,q)$ has one of the  norm sequences in Table~\ref{tab:gofklist} 
 with corresponding exponent sum for some integers $r,s\geq2$.

\begin{table}[h!]
\begin{tabular}{c||c|c|c|c|c|c|c}
Norm Seq.\   & $(r,2,s)$ & $(r)$  & $(r,3)$ & $(r,3,2^{[s-1]})$ & $(2^{[r-1]})$ & $(4,2^{[s-1]})$ & $(2^{[r-1]},4,2^{[s-1]})$ \\ 
\hline
\begin{tabular}{c} Exponent Sum \\ of GOFK \end{tabular} & $r+s-1$ & \begin{tabular}{c} $r\pm1$,  \small{or} \\ $-3$ if $r=4$\end{tabular} & $r-2$ & $r-s-1$ & $-r\pm1$ & $-r-2$ & $-r-s-1$
\end{tabular}
\caption{}
\label{tab:gofklist}
\end{table}

\end{lemma}

\begin{remark}
Note the norm sequences $(r)$ and $(2^{[r-1]})$ each have two exponent sums. These correspond to the genus one fibered knots obtained by positive and negative Hopf bands onto an annular open book.  The norm sequence $(4)$ has a third exponent sum associated to it since $L(4,1)$ has a third genus one fibered knot.
\end{remark}

\begin{proof}
If $L(p,q)$ contains a genus one fibered knot,  then the corresponding two-bridge link has braid index at most $3$ \cite{baker-cgofkils}.   Any $3$--string braid whose closure is a two-bridge link is conjugate to the braid $\sigma_1^a \sigma_2^{-2} \sigma_1^b \sigma_2$ for some integers $a,b$ as one may observe from Murasugi's classification of two-bridge links with braid index $3$ \cite{Murasugi1991braid} (see also \cite{Stoimenow2006ThreeBraids}) and Birman-Menasco's study of multiple presentations of links as closed $3$ braids \cite{BM-linksviaclosedbraidsIII}.
  (Also note that the mirror $\sigma_1^{-a} \sigma_2^2 \sigma_1^{-b} \sigma_2^{-1}$ is conjugate to  $\sigma_1^{-a-1} \sigma_2^{-2} \sigma_1^{-b-1} \sigma_2$.)    This braid has exponent sum $a+b-1$, and its closure may be isotoped into the two-bridge plat closure of the braid $\sigma_1^a \sigma_2^{2} \sigma_1^b$.   Hence this two-bridge link has the associated continued fraction $[a,2,b]$  for some integers $a,b$.  This gives the first norm sequence $(a,2,b)$ in Table~\ref{tab:gofklist} if both $a,b \geq 2$.   

So assume it is not the case that both $a,b \geq 2$. By reversal of the sequence we may assume both $a \geq b$ and $1 \geq b$.    The (weak) norm sequences for these cases are given in the chart below from which one may produce the rest of Table~\ref{tab:gofklist}.   Recall that the sequences $(1)$ and $()$ correspond to $S^3$ while $(0)$ corresponds to $S^1 \times S^2$.

\[
\begin{array}{c||cccc}
(a,2,b)   & b=1 & b=0 & b=-1 & b=-c\leq2 \\
\hline
a\geq 2  & (a-1) & (a) & (a,3) & (a,3,2^{[c-1]}) \\
a=1       & (0)    & (1) & (2) & (2^{[c]}) \\
a=0      &           &(0)  & () & (2^{[c-1]}) \\
a=-1    &             &     &(4) & (4,2^{[c-1]}) \\
a=-d\leq2 &          &    &     &(2^{[d-1]},4,2^{[c-1]})
\end{array}
\]
\end{proof}

\begin{lemma}\label{lem:gofklens}
If a lens space both contains a genus one fibered knot and may be obtained by positive integral surgery on a knot in $S^3$, then up to orientation preserving homeomorphism the lens space $L(p,q)$ and the homology class $k \mod p$ of some orientation of the surgery dual knot satisfy
\[ (p,q,k) \in \{(n,1,1), (7,3,2), (13,4,3), (13,9,2), (18,11,5), (19,3,4), (27,11,4), (32,7,5), (9t+14,-9,3) \}\]
where $n$ and $t$ range over the integers.
\end{lemma}



\begin{proof}
Greene lists all the norm sequences of lens spaces that may be obtained by positive surgery on some knot in $S^3$.  In particular he divides them into large types and small types.  We refer the reader to Sections 9.3 and 9.4 \cite{Greene2013Realization} for their descriptions.  

For the small types, one may compare the sequences $\nu$ in Tables 2 and 3 \cite{Greene2013Realization} directly with the sequences of Lemma~\ref{lem:gofksurgerysequences}. Most all of the sequences in these tables have three or more elements that are $3$ or greater (for any valid choice of $a,b,c$), and thus cannot appear in the Table~\ref{tab:gofklist}.   Let us list those that may have at most two elements that are $3$ are greater by their associated Proposition in Tables 2 and 3:  6.4(5), 6.5(1), 6.5(2), 7.5(3), 8.7(2), 8.7(3), 8.8(1), 8.8(3).   Among these, one then finds only four sequences also appearing in Table~\ref{tab:gofklist}.  The sequences and their associated lens space $L(p,q)$ and homology class $k \mod p$ are in Table~\ref{tab:smalltypes}.  
\begin{table}[h]
\begin{tabular}{c||c|c}
Prop in \cite{Greene2013Realization} & Sequence & Lens space and homology class\\
\hline 
6.5(1) &  $(2^{[n]})$, $n\geq 1$ & $L(n+1,1)$, $k=1$ \\
6.5(2) & $(2,2,3,5)$  & $L(32,7)$, $k=5$ \\
8.7(2) & $(4,3,2)$ & $L(18, 11)$, $k=5$ \\
8.7(3) & $(2,3,4)$ & $L(18, 5)$, $k=7$
\end{tabular}
\caption{Norm sequences of small type lens spaces containing genus one fibered knots.}
\label{tab:smalltypes}
\end{table}
An orientation preserving homeomorphism relates the last two.

For the large types, Greene gives the following six norm sequences which we list in Table~\ref{tab:largetypes} by the associated proposition number in \cite{Greene2013Realization}.
\begin{table}[h]
\begin{tabular}{c||c}
Prop in \cite{Greene2013Realization} & Sequence\\
\hline
6.5(3) & $(a_1, \dots, a_\ell, 2, b_m, \dots, b_2)$ \\
8.3(1),(3) & $(a_1, \dots, a_\ell+b_m, \dots, b_2)$ \\
8.3(2)&$(a_1, \dots, a_\ell, 5, b_m, \dots, b_2)$ \\
6.2(3) &$(a_1, \dots, a_\ell+1, 2, 2, b_m + 1, \dots, b_2)$ \\
7.5(1)& $(a_1, \dots, a_\ell, b_m, \dots, b_1)$ \\
8.2& $(a_1, \dots, a_\ell+b_m+1,\dots,b_1)$.
\end{tabular}
\caption{Norm sequences of large type lens spaces}
\label{tab:largetypes}
\end{table}
For each, the integers  $a_1, \dots, a_\ell\geq 2$ and $b_1, \dots, b_m \geq 2$ satisfy
$1/[a_1, \dots, a_\ell] + 1/[b_1, \dots, b_m] = 1$.  
Note that a choice of integers $a_1, \dots, a_\ell \geq2$ determines the integers $b_1, \dots, b_m \geq2$ by the Riemenschneider {\em point rule} \cite{riemenschneider} (or equivalently by taking a ``dual path'' in the Farey Tessellation).
By the point rule, exactly one of $a_\ell$ and $b_m$ equals $2$.

We split the argument into the cases where $m=1$ and $m>1$.  For the lens spaces $L(p,q)$ obtained here, the homology class $k \mod p$ of the surgery dual may be determined from the property that $-k^2 =q \mod p$.

{\bf Case $m=1$:}   Then $b_m=b_1$ so that $(a_1, \dots, a_\ell) = (2^{[b_1-1]})$, and we have six types of sequences.  To the right of them, we give any constraints on $b_1\geq 2$ needed to produce a sequence in Table~\ref{tab:gofklist}, the resulting sequence, and the associated lens space.
\begin{table}[h]
\begin{tabular}{c||c|c|c|c}
Prop in \cite{Greene2013Realization} & Sequence & Constraint & Sequence in Table~\ref{tab:gofklist}& Lens Space\\
\hline
6.5(3) &  $(2^{[b_1]})$ & $b_1\geq 2$ & $(2^{[b_1]})$ & $L(b_1+1,1)$\\
8.3(1),(3) & $(2^{[b_1-2]})$ & $b_1 \geq 3$ & $(2^{[b_1-2]})$ & $L(b_1-1,1)$\\
8.3(2)& $(2^{[b_1-1]},5)$ & $b_1 = 3$ & $(2,2,5)$ & $L(13,9)$\\
6.2(3) & $(2^{[b_1-2]},3,2,2)$ & $b_1=2$ & $(3,2,2)$ & $L(7,3)$ \\
7.5(1) & $(2^{[b_1-1]},b_1)$ & $b_1=3$, $b_1=4$ & $(2,2,3)$, $(2,2,2,4)$ & $L(7,5)$, $L(13,10)$\\
8.2 & $(2^{[b_1-1]},b_1+3)$ & $b_1=3$ & $(2,2,5)$ & $L(13,9)$
\end{tabular}
\caption{Norm sequences with $m=1$ of large type lens spaces containing a genus one fibered knot.}
\label{tab:largetypem1}
\end{table}

{\bf Case $m>1$:}  

Let us observe two things for a sequence $(c_1, \dots, c_n)$ in Table~\ref{tab:gofklist}.
First, if $c_i=c_{i+1}=2$, then either $c_j=2$ for all $j <i$ or $c_j=2$ for all $j>i+1$ as well.
Next, if $c_i \neq 2$, then one of the following occurs:
\begin{itemize}
\item $i=1$ or $n$,
\item $c_i=4$ and $c_j=2$ for $j\neq i$, or
\item $c_i=3$ and either $i=2$ or $i=n-1$ and $c_j=2$ for $j>2$ or $j<n-1$ respectively.
\end{itemize}

We now examine the implications of these for the large type lens space surgery sequences listed in Table~\ref{tab:largetypes}.

\smallskip
\noindent Type $(a_1, \dots, a_\ell, 2, b_m, \dots, b_2)$:

If $a_\ell =2$ and $b_m \neq 2$, then for the sequence to be in Table~\ref{tab:gofklist} we must have $(a_1, \dots, a_\ell) = (2^{[t]})$ for some integer $t$ due to the consecutive $2$s.  The point rule implies $m=1$, a contradiction.

If $a_\ell \neq 2$ and $b_m=2$, then the same reasoning implies $\ell =1$ and $(b_1, \dots, b_m) = (2^{[a_1-1]})$.  Thus the sequence is $(a_1, 2, 2^{[a_1-2]})$.  For this to be in Table~\ref{tab:gofklist}, we must have either $a_1=3$ or $a_1=4$ which give the sequences $(3,2,2)$ and $(4,2,2,2)$ respectively.  The associated lens spaces are $L(7,3)$ and $L(13,4)$.

\smallskip
\noindent  Type $(a_1, \dots, a_\ell+b_m, \dots, b_2)$:

Since $a_\ell+b_m \geq 5$,  for the sequence to be in Table~\ref{tab:gofklist} either $\ell=1$ or $m=2$.  If $\ell =1$ then the point rule implies $(b_1, \dots, b_m) = (2^{[a_1-1]})$ so that the sequence is $(a_1+2,2^{[a_1-3]})$.  Hence $a_1=5$, yielding the sequence $(7,2,2)$ and associated lens space $L(19,3)$.

So assume $\ell \neq 1$ and $m=2$. Then since  $a_\ell+b_2 \geq 5$ and $\ell \neq 1$ we have $(a_1, \dots, a_\ell) = (a_1,2,a_3)$ or $(2^{[t]},3,a_\ell)$.   By the point rule and $m=2$, the former implies $(a_1, 2, a_3)=(2,2,3)$ or $(3,2,2)$ with $(b_1, b_2) = (4,2)$ or $(2,4)$ respectively.  Thus the sequence is either $(2,2,5)$ or $(3,2,6)$ respectively. Their associated lens spaces are $L(13,9)$ and $L(27,11)$

 For the latter, the point rule and $m=2$ imply $a_\ell=2$ and $(b_1,b_2) = (t+1, 3)$ with $t\geq 1$.  The sequence is then $(2^{[t]},3,5)$ and its associated lens space is $L(9t+14,-9)$.   (Here, $9t+14=ik-1$ with $k=3$.)

\smallskip
\noindent  Type $(a_1, \dots, a_\ell, 5, b_m, \dots, b_2)$:

Since $m>1$, the $5$ cannot be at the beginning or end of the sequence.  This sequence cannot be in Table~\ref{tab:gofklist}.

\smallskip
\noindent  Type $(a_1, \dots, a_\ell+1, 2, 2, b_m+1, \dots, b_2)$:

The middle pair of $2$'s imply that either $a_\ell = 1$ or $b_m =1$.  This is a contradiction.

\smallskip
\noindent  Type $(a_1, \dots, a_\ell, b_m, \dots, b_1)$:

By reversal and appealing to Case $m=1$, we may assume both $m>1$ and $\ell>1$.  Also by reversal, we may assume $a_\ell=2$ and $b_m\neq2$.  Since $m>1$, $(a_1, \dots, a_\ell) \neq (2^{[t]})$ for any $t>0$.  Thus $a_k>2$ for some $k<\ell$.  Since $a_\ell=2$ is between $b_m$ and $a_k$, for the sequence to be in Table~\ref{tab:gofklist} we must have $\ell=2$ and $m=1$.  This is a contradiction.

\smallskip
\noindent  Type $(a_1, \dots, a_\ell+b_m+1,\dots, b_1)$:

Since $a_\ell + b_m +1 \geq 6$, the argument for Type (3) applies.
\end{proof}

\subsection{Proof of Theorem~\ref{thm:altGOFKsurg}}

\begin{proof}
Let $L(p,q)$ be the lens space obtained by positive doubly primitive surgery on a Berge knot in family VII or VIII.
Assume the lens space $L(p\pm1,q')$ can be obtained as an alternative surgery on this knot.  By Lemma~\ref{lem:gofktool}(3), it contains a genus one fibered knot.
By Lemma~\ref{lem:gofktool}(4)  this lens space may only be among the lens spaces listed in Lemma~\ref{lem:gofklens}.  Lemma~\ref{lem:gofktool}(1) and (2) reduces this list to the lens spaces $L(18,11)$, $L(32,7)$, and $L(18t'+14,-9)$ for $t'\geq1$.   

If $L(p\pm1,q') =L(18,9)$ then $p=19$.  (The proof of Lemma~\ref{lem:gofktool}(1) shows $p\neq17$.)  The only non-torus knot in family VII and VIII for which $19$--surgery yields a lens space is the pretzel knot $P(-2,3,7)$.  This knot has the required $18$--surgery.

If $L(p\pm1,q') = L(32,7)$, then either $p=31$ or $p=33$.  Equation ($\star$) has no solution when $p=33$.  When $p=31$, equation ($\star$) has the solutions $(p,q,k) = (31,6,5)$, $(31,26,25)$, $(31,17,18)$, and $(31,11,12)$.  By Lemma~\ref{lem:simpleknotequiv}, the first two give orientation preserving homeomorphic knots.  These are torus knots (e.g.\ by Saito \cite{saito2008dual}).  
The last two knots are also orientation preserving homeomorphic by Lemma~\ref{lem:simpleknotequiv}.  There the corresponding knot is surgery dual to the knot $\WSL_{+1}$ which also has the required $32$--surgery.

For $L(p\pm1,q')=L(18t'+14,-9)$ with $t'\geq1$, the lens space has the positive surgery description  $(2t'+2,3,2,2,2)$.  From Table~\ref{tab:gofklist}, a genus one fibered knot in this lens space has exponent sum $2t'-3$.  Lemma~\ref{lem:gofktool}(3) then implies $t' \in \{1,2,3\}$ giving the lens spaces $L(32,-9)$, $L(50,-9)$, and $L(68,-9)$.  Since $L(32,-9)$ is orientation preserving homeomorphic to $L(32,7)$ which we addressed in the previous case, we only need to handle the last two.

If $p\pm1=50$, then either $p=49$ or $p=51$.  Equation ($\star$) has no solution when $p=51$.  When $p=49$, equation ($\star$) has the solutions $(p,q,k) = (49,19,18)$ and $(49,31,30)$.  These correspond to orientation preserving homeomorphic knots by Lemma~\ref{lem:simpleknotequiv} that have genus $17$ by Theorem~\ref{thm:simpleknoteulerchar}.  However $L(50,-9)$ contains no primitive simple knots of this genus which we also calculate by Theorem~\ref{thm:simpleknoteulerchar}, so this case does not occur.

If $p\pm1=68$, then either $p=67$ or $p=69$.   Equation ($\star$) has no solution when $p=69$.  When $p=67$, equation ($\star$) has the solutions $(p,q,k) = (67,30,29)$ and $(67,38,37)$.  These correspond to orientation preserving homeomorphic knots of genus $25$.  However $L(68,-9)$ contains no primitive simple knots of this genus, so this case does not occur.  Again we use Lemma~\ref{lem:simpleknotequiv} and Theorem~\ref{thm:simpleknoteulerchar}.
\end{proof}

\section{Knots with once-punctured torus Seifert surfaces}

In \cite{baker-oncepuncturedtoriandknotsinlensspaces}, the first author classifies the non-nullhomologous knots in lens spaces (other than $S^1\times S^2$) with once-punctured torus surfaces properly embedded in their exterior.  We call such surfaces once-punctured torus Seifert surfaces.  

Let $W$ be the Whitehead link , and let $Y$ be the link that covers $W$ in the double cover of $S^3$ branched over one component of $W$.  These are both pictured in Figure~\ref{fig:WandY}.  Let $K^s_r$ be the core of $r$--surgery in $W(r,s)=W(s,r)$.  Let $K^{a,b}_{c}$ be the core of the $c$--surgery in $Y(a,b;c) = Y(b,a;c)$ where the $c$--surgery is on the ``axis'' component of $Y$.  (This was written as $Y(c;a,b)$ in  \cite{baker-oncepuncturedtoriandknotsinlensspaces}.)

\begin{figure}
\centering
\psfrag{W}{$W$}
\psfrag{Y}{$Y$}
\includegraphics[height=1in]{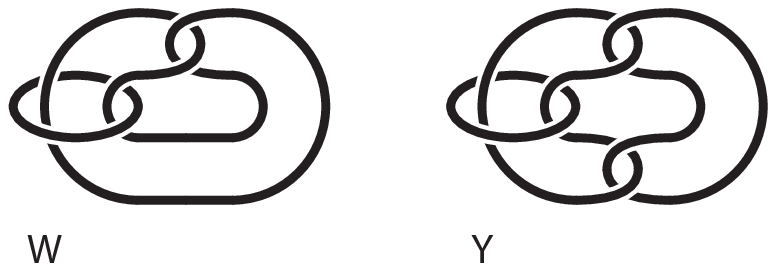}
\caption{}
\label{fig:WandY}
\end{figure}

\begin{theorem}[Baker \cite{baker-oncepuncturedtoriandknotsinlensspaces}]\label{thm:optknots}
If $K$ is a non-nullhomologous knot in a lens space with an incompressible once-punctured torus Seifert surface, then up to homeomorphism $K$ is one of the following knots:
\begin{enumerate}
\item  $K_{-6+1/k}^{-1} \subset L(6k-1,2k-1)=W(-1, -6+1/k)$,
\item  $K_{-4+1/k}^{-2} \subset L(8k-2,2k-1)=W(-2, -4+1/k)$,
\item  $K_{-3+1/k}^{-3} \subset L(9k-3,3k-2)=W(-3, -3+1/k)$,
\item  $K^{-3+1/k}_{-3} \subset L(9k-3,3k-2)=W(-3+1/k,-3)$ for $k \neq 0$, or
\item  $K^{1/k,1/\ell}_{-2} \subset L(8k\ell-2,4k\ell-2k-1)=Y(1/k,1/\ell;-2)$ for $k,\ell \neq 0$.
\end{enumerate}
Furthermore:
\begin{enumerate}
\item each of these knots is a simple knot,
\item the first three families are all torus knots and are thus fibered,
\item $K^{-3+1/k}_{-3}$ is fibered only if $|k|=1$, and
\item $K^{1/k,1/\ell}_{-2}$ is fibered only if $|k|=|\ell|=1$.
\end{enumerate}
\end{theorem}

\begin{remark}
As noted in \cite{baker-oncepuncturedtoriandknotsinlensspaces}, observe that
\begin{enumerate}
\item $K^{1/k,1}_{-2} = K^{-4+1/k}_{-2} \subset L(8k-2,2k-1)$ and
\item $K^{1/k,-1}_{-2}$ is the mirror of $K^{-4-1/k}_{-2}$.
\end{enumerate}
\end{remark}

\begin{remark} 
The exterior of $W$ is hyperbolic and  orientation reversing homeomorphic to $M_3(-1)$.  The exterior of $Y$ is also hyperbolic.  Hence the knots $K^{-3+1/k}_{-3}$ and $K^{1/k,1/\ell}_{-2}$ of the last two families of Theorem~\ref{thm:optknots} are generically hyperbolic.
\end{remark}

From this theorem we easily obtain a classification of non-nullhomologous knots with once-punctured torus Seifert surfaces admitting lens space surgeries.

\begin{theorem}\label{thm:optsurg}
The following families of pairs of knots are surgery dual for every $k, \ell \in \Z$.
\begin{enumerate}
\item  $K^{-1}_{-6+1/k}\subset L(6k-1,2k-1)$ and $K^{-1}_{-6+1/\ell} \subset L(6\ell-1,2\ell-1)$ 
\item $K^{-2}_{-4+1/k} \subset L(8k-2,2k-1)$ and $K^{-2}_{-4+1/\ell} \subset L(8\ell-2,2\ell-1)$
\item $K^{-3}_{-3+1/k} \subset L(9k-3,3k-2)$ and $K^{-3}_{-3+1/\ell} \subset L(9\ell-3,3\ell-2)$
\item $K^{-3+1/k}_{-3} \subset L(9k-3,3k-2)$ and $K^{-3+1/k}_{\infty} \subset L(3k-1,-k)$
\item $K^{-4+1/k}_{-2} \subset L(8k-2,2k-1)$ and $K^{-4+1/k}_{\infty} \subset L(4k-1,-k)$
\item $K^{-6+1/k}_{-1} \subset L(6k-1,2k-1)$ and 
$K^{-6+1/k}_{\infty} \subset L(6k-1,-k)$
\end{enumerate}
Every non-nullhomologous knot in a lens space other than $S^1 \times S^2$ with an incompressible once-punctured torus Seifert surface and a non-trivial lens space surgery belongs to one of the first five families above. 
\end{theorem}

\begin{remark}\label{rem:nonfibered}\
\begin{enumerate}
\item The surgery duals in the fourth and fifth families consist of a null-homologous and non-nullhomologous pair. 
\item The surgery duals in the sixth family consist of pair of nullhomologous genus one knots.  
\item In these last three families, the knots are not fibered if $|k|>1$ (since only integral filling of one component of the exterior of $W$ can produce a once-punctured torus bundle).
\end{enumerate}
\end{remark}

\begin{prob}
Classify lens space surgeries on nullhomologous genus one knots in lens spaces.
\end{prob}

Remark~\ref{rem:nonfibered}(3) gives the following corollary.
\begin{cor}\label{cor:nonfibered}
There are infinitely many non-fibered hyperbolic knots in lens spaces with non-trivial lens space surgeries. \qed
\end{cor}

\begin{remark}
Corollary~\ref{cor:nonfibered} should be contrasted with Theorem~6.5 of \cite{BBCW} (attributed to Rasmussen) which shows that any primitive knot in a lens space with a non-trivial lens space surgery has fibered exterior.  Here, a knot is primitive if, when oriented, it represents a generator of homology. 
\end{remark}

\begin{proof}[Proof of Theorem~\ref{thm:optsurg}]
The list of surgery duals is immediate from Theorem~\ref{thm:optknots}.  The first three families are torus knots and these are the lens space surgeries on torus knots.  The last three families result from $W$ being a link of unknots so that the trivial filling on one component is necessarily a lens space.  To see that these are all the lens space surgeries on such non-nullhomologous knots it remains to address the lens space surgeries on the knots $K^{-3+1/k}_{-3}$ and $K^{1/k,1/\ell}_{-2}$.  When these knots are not torus knots, any lens space surgery must be distance $1$.  Furthermore, by the classification of Theorem~\ref{thm:optknots} such surgery slopes must have distance at most $3$ from the boundary slope of the Seifert surface.

For $K^{-3+1/k}_{-3}$ this means that the dual knot is either $K^{-3+1/k}_{-2}$ or $K^{-3+1/k}_{\infty}$.  The former is not a knot in a lens space.   

For $K^{1/k,1/\ell}_{-2}$ the dual knot must be either $K^{1/k,1/\ell}_{-3}$, $K^{1/k,1/\ell}_{-3/2}$, $K^{1/k,1/\ell}_{-1}$, or $K^{1/k,1/\ell}_{\infty}$.  
Since the first two duals have order $3$, if they are indeed knots in lens spaces they would have to be the knot $K^{-3+1/m}_{-3}$ or its mirror for some $m\neq 0,1$. However, by the former case, such knots are not dual to knots in lens spaces of order $2$.  For the last two duals, we must determine for which $k,\ell$ are the manifolds $Y(1/k,1/\ell;-1)$ and $Y(1/k,1/\ell; \infty)$ lens spaces.  

First note that, by performing the $-1$ and $\infty$ surgeries, $Y(-1/k,-1/\ell;-1)$ is the mirror of $Y(1/k, 1/\ell;\infty)$.  Hence we may work with the latter of the two.
Observe that this is $1/k$ and $1/\ell$ surgery on the $(2,4)$--torus link.  Performing the $1/\ell$ surgery gives the $(2, 1-2\ell)$--torus knot with surgery coefficient $-4\ell+1/k$.  If $\ell=0,1$ then this torus knot is actually an unknot and any $k$ gives a lens space.  But recall we discard $\ell=0$ since the once-punctured torus would compress and if $\ell=1$, then $K^{1/k,1}_{-2} = K^{-4+1/k}_{-2} \subset L(8k-2,2k-1)$.  If $\ell \neq 0,1$ then the $(2, 1-2\ell)$--torus knot  is not the unknot and only $2-4\ell +1/n$ surgery for $n \in \Z$ yields a lens space.  Hence we must have $-4\ell+1/k = 2-4\ell+1/n$ and thus $k=\pm1$.  With $k=1$ we have the knot $K^{-4+1/\ell}_{-2}$ as before (when we had $\ell=1$).  With $k=-1$ we have the mirror of $K^{-4-1/\ell}_{-2}$.
\end{proof}

\begin{cor}\label{cor:figeightknottriple}
The only non-torus, non-trivial knot exterior with an once-punctured torus Seifert surface and three lens space fillings is the Figure Eight Knot Sister manifold, $K^{-5} = W(-5,\cdot)$. Its three lens space fillings are given in the fifth and sixth family of Theorem~\ref{thm:optsurg} with the knots $K^{-5}_{-2} \subset L(10,3)$, $K^{-5}_{\infty} \subset L(5,-1)$, and $K^{-5}_{-1} \subset L(5,1)$ as cores of the fillings.  
\end{cor}

\begin{proof}
By the Cyclic Surgery Theorem \cite{CGLS}, if a manifold has three lens space fillings then either the manifold is reducible or a Seifert fibered space or the three fillings are mutually distance $1$.  By Theorem~\ref{thm:nonhyperbolic} the only non-hyperbolic manifolds with at least three lens space fillings are the exteriors of torus knots and trivial knots.  If the lens space knots $K_1,K_2,K_3$ are duals to one another by distance $1$ surgeries, then all three cannot be nullhomologous.  Thus, assuming they are not torus knots, up to homeomorphism and reindexing we may take $K_1$ to be the knot $K^{-3+1/k}_{-3}$ or $K^{-4+1/k}_{-2}$ for some integer $k\neq 0,1$. These have surgeries to the nullhomologous knots $K^{-3+1/k}_{\infty}$ and $K^{-4+1/k}_{\infty}$ respectively as described in Theorem~\ref{thm:optsurg}.  From the proof of that theorem, the knots $K^{-3+1/k}_{-3}$ do not have any other lens space surgery when $k\neq 0,1$.  The same proof also shows that among the knots $K^{1/k,1/\ell}_{-2}$, only the knot $K^{-1,1}_{-2} = K^{1,-1}_{-2}=K^{-5}_{-2}$ has the extra lens space surgery $K^{-5}_{-1}$. The manifold $K^{-5} = W(-5,\cdot)$ is the Figure Eight Knot Sister manifold.
\end{proof}

\bibliographystyle{alpha}  
\bibliography{ThreeSurgeriesBib}

\end{document}